\documentclass[a4paper,10pt]{article}

\usepackage{a4wide}
\usepackage[utf8]{inputenc}
\usepackage[english]{babel}
\usepackage{amsmath,amssymb,amsfonts,amsthm,amsbsy}
\usepackage{mathrsfs}
\usepackage{mathtools}
\usepackage[font=footnotesize,labelfont=footnotesize]{caption}
\usepackage{authblk}
\usepackage{subcaption}
\usepackage{multirow}
\usepackage{comment}
\usepackage{bm}
\mathtoolsset{showonlyrefs}

\usepackage{tikz}
\usepackage{pgfplots}
\pgfplotsset{every axis/.append style={
		label style={font=\scriptsize},
		tick label style={font=\scriptsize},
		legend style={font=\tiny\sf, row sep=-2pt},
		legend cell align={left}
}}
\newtheorem{theorem}{Theorem}[section]
\newtheorem{lemma}[theorem]{Lemma}
\newtheorem{corollary}[theorem]{Corollary}

\newtheorem{problem}[theorem]{Problem}

\newtheorem{assumption}[theorem]{Assumption}

\theoremstyle{remark}

\numberwithin{equation}{section}

\usepackage[numbers,sort&compress]{natbib}

\usepgfplotslibrary{external}
\tikzexternalenable

\usepackage[dvipsnames]{xcolor}
\usepackage[colorlinks=true,citecolor=red!75!green,linkcolor=blue!75!green]{hyperref}
\definecolor{ForestGreen}{RGB}{34,139,34}

\newcommand{\vecalg}[1]{\boldsymbol{\mathsf{#1}}}
\newcommand{\matalg}[1]{\boldsymbol{\mathsf{#1}}}

\renewcommand{\Re}{\operatorname{Re}}

\newcommand{\im}{\imath}
\renewcommand{\i}{\im}

\newcommand{\grad}{\boldsymbol \nabla}
\renewcommand{\div}{\grad \cdot}
\newcommand{\curl}{\grad \times}

\newcommand{\TS}{\textup{S}}
\newcommand{\TI}{\textup{I}}

\title{A HDG method with transmission variables for time-harmonic wave propagation problems with constant coefficients}% and application to aeroacoustics with a uniform mean flow}

\author[1]{S.~Pescuma}
\author[2]{G.~Gabard}
\author[3]{T.~Chaumont-Frelet}
\author[4]{A.~Modave}

\affil[1]{\footnotesize CERMICS, CNRS, ENPC, Institut Polytechnique de Paris, Marne-la-Vallée, France,
	\href{mailto:simone.pescuma@enpc.fr}{\texttt{simone.pescuma@enpc.fr}}}
\affil[2]{Laboratoire d'Acoustique de l'Université du Mans (LAUM), IA-GS, CNRS, Le Mans, France
	\href{mailto:gwenael.gabard@univ-lemans.fr}{\texttt{gwenael.gabard@univ-lemans.fr}}}
\affil[3]{\footnotesize Inria, Laboratoire Paul Painlevé, Université de Lille, Villeneuve-d'Ascq, France
	\href{mailto:theophile.chaumont@inria.fr}{\texttt{theophile.chaumont@inria.fr}}}
\affil[4]{\footnotesize POEMS, CNRS, Inria, ENSTA, Institut Polytechnique de Paris, Palaiseau, France,
	\href{mailto:axel.modave@ensta.fr}{\texttt{axel.modave@ensta.fr}}}
\date{}

\begin{document}

\maketitle

\begin{abstract}
Iterative finite element solvers for time-harmonic wave problems are notoriously slow to converge, owing to fundamental properties of these problems.
We present a variant of the hybridizable discontinuous Galerkin (HDG) method that is better suited to fast iterative solution.
Unlike the standard hybridization strategy, which eliminates physical unknowns by introducing an auxiliary numerical flux on element faces, our approach instead introduces a transmission variable on those faces.
For Helmholtz problems, this reformulation, known as CHDG, has been shown to significantly accelerate the convergence of iterative schemes relative to standard HDG.
The present work extends CHDG to a general framework covering wave propagation problems with constant coefficients, capable of handling diverse wave types in a unified manner.
We prove that the resulting hybridized system is well-posed and amenable to fixed-point iteration.
As a practical application, we apply the method to the time-harmonic linearized Euler equations with a uniform subsonic mean flow.
The method is illustrated through two-dimensional numerical benchmarks involving both sound and vorticity waves, with a systematic comparison of the convergence behaviour of several iterative schemes across a range of configurations.
\end{abstract}

\setlength{\parskip}{4pt}

\section{Introduction}

Wave propagation appears in a wide range of phenomena, such as acoustics, aero-acoustics, electromagnetism, elasticity and water waves.
The difficulty in these problems is that the wave fields can be highly oscillatory and can cover large domains.
Discontinuous Galerkin (DG) finite element methods are particularly well suited to computing high-fidelity numerical solutions in situations involving complex geometries and physics \cite{shi2018, lu2004, carstensen2016, hesthaven2007, klockner2009, chaumontfrelet2024, imbertgerard2024}.
Furthermore, in the frequency domain, using high-order polynomial basis functions reduces dispersion error in high-frequency cases \cite{beriot2015, lieu2016, christodoulou2017, congreve2019}.

From a computational point of view, the finite element discretization of time-harmonic problems leads to sparse unstructured linear systems.
Direct solvers can become expensive and difficult to run efficiently on parallel architectures.
While iterative solvers require far less memory and allow for efficient parallel implementations, their convergence may be slow due to the intrinsic properties of the linear system itself, see for instance \cite{ernst2011difficult}.
Preconditioning techniques and domain decomposition methods have been, and continue to be, actively studied in order to accelerate convergence \cite{bayliss1983, erlangga2004, engquist2011, gander2019, dolean2008, bouajaji2012}.

We consider a class of DG methods called hybridizable discontinuous Galerkin (HDG) methods \cite{barucq2021_2, barucq2023, chen2013, cockburn2008, nguyen2011, li2013, camargo2020}.
With this type of methods, a hybrid variable is defined on the mesh skeleton, which allows the physical variables to be decoupled at the interface between elements.
The physical variables defined within each element can then be eliminated.
This yields a smaller, hybridized system involving only the hybrid variable itself.
The physical variables within each element are easily recovered by a quick, and local, post-processing step.
Hybridization not only modifies the size and structure of the global system, but also its conditioning.
This opens the door to more efficient iterative solution procedures.

With the classical HDG approach, the hybrid variable corresponds to a numerical trace.
Here, the hybrid variable corresponds to transmission variables.
They are directly related to the physics of the problem and are based on the concept of characteristic variables.
This approach, called CHDG, was initially developped for Helmholtz problems in~\cite{modave2023} by two of the authors.
The strategy can be linked to some non-overlapping domain decomposition methods \cite{collino2020, despres1991, farhat2009, modave2020} and Trefftz methods based on ultra weak variational formulations \cite{barucq2021, barucq2024, cessenat1998, gabard2007, huttunen2002, parolin2022, rivet2026, imbertgerard2025, imbertgerard2025b, pernet2025}.
In \cite{modave2023}, numerical results show that standard iterative schemes converge faster when applied to the CHDG system than to systems obtained using the non-hybridized DG method or the standard HDG method.
Similar results were obtained in \cite{rappaport2025} for electromagnetic waves in homogeneous media and in \cite{pescuma2025} for acoustic waves in heterogeneous media.

The present work extends the CHDG method to a general category of wave propagation problems in the frequency domain.
More specifically, we consider linear symmetric hyperbolic systems with constant coefficients.
This encompasses the previous CHDG methods considered in \cite{modave2023, rappaport2025}.
The DG scheme is defined using upwind fluxes, which are written in terms of transmission variables.
The scheme is then hybridized using the incoming transmission variable, leading to a hybridized system of the form $(\TI-\Pi\TS)\bm{G}^-=\bm{b}$, with the set of transmission variables $\bm{G}^-$, an exchange operator $\Pi$, a scattering operator $\TS$, and a right-hand side $\bm{b}$.
Assuming the boundary conditions are passive, we prove that the operator $\Pi\TS$ is a strict contraction, and consequently that the hybridized system can be solved using the fixed-point iteration.

In the second part of this work, the method is applied to the linearized Euler equations to model the propagation of sound and vorticity waves in a uniform subsonic mean flow.
The presence of vorticity waves brings an interesting aspect to the problem, in that they are only convected by the mean flow.
As a consequence, the number of transmission variables on each element face can vary depending on its orientation relative to the mean flow.
The mathematical framework proposed in this work handles this added complexity naturally.
As in previous works \cite{modave2023, pescuma2025, rappaport2025}, the convergence of classical iterative schemes (fixed point, CGNR, and GMRES) applied to the CHDG hybridized system is studied using two-dimensional benchmarks.
Note that HDG methods have been proposed and studied for the convected Helmholtz and Galbrun equations e.g.~in \cite{rouxelin2021, barucq2023, halla2025hybrid}.
The proposed approach differs in that it models both sound and vorticity waves using the linearized Euler equations.

The remainder of this paper is organised as follows.
In Section~\ref{sec:HypSysAndTransVar}, the relevant properties of hyperbolic systems are recalled, and the notion of transmission variables is introduced.
The DG method with upwind fluxes is described in Section~\ref{sec:DG} for a general time-harmonic problem, and the hybridization with transmission variables is explained and analyzed in Section~\ref{sec:CHDG}.
In Section~\ref{sec:aeroac}, we apply the method to the propagation of aeroacoustic waves in a uniform subsonic mean flow.
The method is assessed systematically in Section~\ref{sec:numStudy} by considering a series of two-dimensional benchmark problems with several iterative schemes.
Section~\ref{conclusion} provides conclusions and perspectives.

%%%%%%%%%%%%%%%%%%%%%%%%%%%%%%%%%%%%%%%%%%%%%%%%%%

\section{Hyperbolic systems and transmission variables}
\label{sec:HypSysAndTransVar}

Before discussing their numerical solutions in the frequency domain, we introduce key properties of linear, symmetric, hyperbolic systems in the time domain.
These are written on a three-dimensional bounded domain $\Omega\subset\mathbb{R}^3$ as follows
\begin{equation}
	\label{eqn:3DHypSys}
	\frac{\partial\tilde{\bm{U}}}{\partial t}
	+ \sum_{j=1}^3 \bm{A}_j\frac{\partial\tilde{\bm{U}}}{\partial x_j}
	= \tilde{\bm{S}}_\mathrm{vol},
\end{equation}
where the unknown is an $l$-dimensional, real, vector field $\tilde{\bm{U}}(t,\bm{x}) : \mathbb{R}^+\times\Omega \rightarrow \mathbb{R}^l$.
The coefficient matrices $\{\bm{A}_j\}_{j=1,2,3}\in\mathbb{R}^{l\times l}$ are constant and symmetric.
On the right-hand side, $\tilde{\bm{S}}_\mathrm{vol}(t,\bm{x}) : \mathbb{R}^+\times\Omega \rightarrow \mathbb{R}^l$ denotes a volume source.
Detailed expositions of the properties of such systems can be found in \cite{gustafsson1995,leveque2002,benzoni2006}.

%%%%%%%%%%%%%%%%%%%%%%%%%%%%%%%%%%%%%%%%%%%%%%%%%%%%%%%%%%%%%%%%%%%%%%%%%%%%%%%%%%%%%%%%

\subsubsection*{Characteristic analysis}

Assuming that the solution $\tilde{\bm{U}}$ varies only along the direction given by a unit vector $\bm{n}\in\mathbb{R}^3$, the system becomes (ignoring the source term)
\begin{align}
	\label{eqn:1DHypSys}
	\frac{\partial\tilde{\bm{U}}}{\partial t}
	+ \bm{F}\frac{\partial\tilde{\bm{U}}}{\partial n}
	= \mathbf{0},
\end{align}
where $n := \bm{x}\cdot\bm{n}$ is the coordinate along the $\bm{n}$ direction.
The \textit{flux matrix}
\begin{align}
	\label{eqn:fluxMatrix}
	\bm{F} := \sum_{j=1}^3 n_{j} \bm{A}_{j}
\end{align}
is real and symmetric, and therefore diagonalizable with real eigenvalues.
We write its eigendecomposition $\bm{F} = \bm{W}\bm{\varLambda}\bm{W}^{\dagger}$, with $^\dagger$ the complex transpose, $\bm{\varLambda}$ the matrix of eigenvalues, and $\bm{W}$ the unitary matrix of eigenvectors.
Using this decomposition in \eqref{eqn:1DHypSys} leads to a set of uncoupled transport equations:
\begin{align}
	\label{eqn:1DTranspSys}
	\frac{\partial(\bm{W}^\dagger\tilde{\bm{U}})}{\partial t} + \bm{\varLambda}\frac{\partial(\bm{W}^{\dagger}\tilde{\bm{U}})}{\partial n} = \mathbf{0}.
\end{align}
The components of $\bm{W}^\dagger\tilde{\bm{U}}$ are called \emph{characteristic variables} and they can be interpreted as quantities transported along the direction $\bm{n}$.
For each component, the corresponding eigenvalue gives the direction and velocity of transport.
They are therefore transported in the positive or negative $\bm{n}$ direction for strictly positive and strictly negative eigenvalues, respectively.
Zero eigenvalues correspond to steady components in the solution.

%%%%%%%%%%%%%%%%%%%%%%%%%%%%%%%%%%%%%%%%%%%%%%%%%%%%%%%%%%%%%%%%%%%%%%%%%%%%%%%%%%%%%%%%

\subsubsection*{Transmission variables}

A central aspect of the numerical model proposed in this paper is the definition of the \emph{transmission variables} $\tilde{\bm{G}}^\pm$.
They are scaled versions of the characteristic variables associated with the strictly positive and strictly negative eigenvalues:
\begin{align}
	\label{eqn:transVar}
	\tilde{\bm{G}}^+
	= \sqrt{\bm{\varLambda}^+}(\bm{W}^+)^{\dagger}\tilde{\bm{U}}
	\qquad\text{and}\qquad
	\tilde{\bm{G}}^-
	= \sqrt{-\bm{\varLambda}^-}(\bm{W}^-)^{\dagger}\tilde{\bm{U}},
\end{align}
where the diagonal matrices $\bm{\varLambda}^+\in\mathbb{R}^{m^+\times m^+}$ and $\bm{\varLambda}^-\in\mathbb{R}^{m^-\times m^-}$ contain only the strictly positive and strictly negative eigenvalues, respectively, with $m^+$ and $m^-$ the numbers of strictly positive and strictly negative eigenvalues.
The rectangular matrices $\bm{W}^+\in\mathbb{R}^{l\times m^+}$ and $\bm{W}^+\in\mathbb{R}^{l\times m^-}$ store the associated unit eigenvectors.
The rationale for using this particular scaling of the characteristic variables is explained below.

Like the characteristic variables, the transmission variables $\tilde{\bm{G}}^+$ and $\tilde{\bm{G}}^-$ correspond to the propagation of information in the positive or negative $\bm{n}$ direction.
If the vector $\bm{n}$ is the outward unit normal on the boundary of a spatial region, then $\tilde{\bm{G}}^+$ and $\tilde{\bm{G}}^-$ are referred to as the \emph{outgoing} and \emph{incoming transmission variables}, respectively.
Note that, even with constant coefficient matrices $\bm{A}_j$, the numbers $m^\pm$ of incoming and outgoing transmission variables can change depending on the orientation $\bm{n}$ of the boundary.

%%%%%%%%%%%%%%%%%%%%%%%%%%%%%%%%%%%%%%%%%%%%%%%%%%%%%%%%%%%%%%%%%%%%%%%%%%%%%%%%%%%%%%%%

\subsubsection*{Conservation relations}

The system in \eqref{eqn:3DHypSys} is a set of conservation equations for the quantities in the vector field $\tilde{\bm{U}}$.
Integrating over the domain $\Omega$ yields the corresponding conservation relation in integral form:
\begin{align}
	\frac{d}{dt}\int_\Omega \tilde{\bm{U}}\:d\Omega
	= - \int_{\partial\Omega} \bm{F}\tilde{\bm{U}}\:d\Gamma
	+ \int_{\Omega} \tilde{\bm{S}}_\mathrm{vol}\:d\Omega,
\end{align}
where $\bm{F}$ is the flux matrix \eqref{eqn:fluxMatrix} associated to the unit outward normal $\bm{n}$ to the boundary $\partial\Omega$.
The vector $\bm{F}\tilde{\bm{U}}$ is therefore the \emph{physical flux} of the quantities $\tilde{\bm{U}}$ across the boundary.

Furthermore, left-multiplying the system \eqref{eqn:3DHypSys} by $\tilde{\bm{U}}^\dagger$ and integrating over the domain  $\Omega$ lead to an energy conservation relation
\begin{align}
	\frac{d}{dt}\int_\Omega \tfrac{1}{2}\tilde{\bm{U}}^\dagger\tilde{\bm{U}}\:d\Omega
	= - \int_{\partial\Omega} \tfrac{1}{2}\tilde{\bm{U}}^\dagger\bm{F}\tilde{\bm{U}}\:d\Gamma
	+ \int_{\Omega} \tilde{\bm{U}}^\dagger\tilde{\bm{S}}_\mathrm{vol}\:d\Omega,
	\label{eqn:consEnergy}
\end{align}
where $\frac{1}{2}\tilde{\bm{U}}^\dagger\tilde{\bm{U}}$ is the \emph{energy density} and $\frac{1}{2}\tilde{\bm{U}}^\dagger\bm{F}\tilde{\bm{U}}$ is the \emph{energy flux}.
The sign of the energy flux obviously depends on the nature of the boundary conditions imposed on $\partial\Omega$.

An important connexion can be made with the transmission variables defined on the boundary of the domain.
Using \eqref{eqn:transVar}, the energy flux can be rewritten as
\begin{align}
	\tfrac{1}{2}\tilde{\bm{U}}^\dagger\bm{F}\tilde{\bm{U}}
	= \tfrac{1}{2}\|\tilde{\bm{G}}^+\|_2^2 - \tfrac{1}{2}\|\tilde{\bm{G}}^-\|_2^2.
	\label{eqn:consEnergyFlux}
\end{align}
This indicates that $\tilde{\bm{G}}^+$ and $\tilde{\bm{G}}^-$ are solely associated to the outgoing and incoming energy flux, respectively.
As a consequence, at a time $t$, the boundary $\partial\Omega$ absorbs energy if the norm of the outgoing transmission variable $\tilde{\bm{G}}^+$ is larger than the norm of the incoming one $\tilde{\bm{G}}^-$.
In the time domain, a boundary is said to be \emph{passive} if the net energy flux is positive at all time $t$, with the net energy flux defined as the time integration of the energy flux between the initial condition and a time $t$ (see for instance \cite{srivastava2021}).

%%%%%%%%%%%%%%%%%%%%%%%%%%%%%%%%%%%%%%%%%%%%%%%%%%%%%%%%%%%%%%%%%%%%%%%%%%%%%%%%%%%%%%%%

\subsubsection*{Boundary conditions}

The characteristic variables also play a central role in defining well-posed boundary conditions on $\partial\Omega$ \cite{higdon1986initial,benzoni2006}.
For linear hyperbolic systems, one should specify $m^-$ boundary conditions (with $m^-$ the number of strictly negative eigenvalues of the flux matrix $\bm{F}$) such that every incoming characteristic variable is prescribed.
In other words, the boundary conditions should be such that the incoming transmission variables $\tilde{\bm{G}}^-$ are fully specified.
Therefore, for a portion $\Gamma$ of the boundary $\partial\Omega$ where $m^-$ is constant, we consider a generic family of linear boundary conditions
\begin{align}
	\label{eqn:BCtime}
	\tilde{\bm{G}}^- = \bm{\alpha} \tilde{\bm{G}}^+
	+ \tilde{\bm{S}}_\mathrm{sur},
\end{align}
where the incoming transmission variables $\tilde{\bm{G}}^-$ are expressed in terms of the outgoing transmission variables $\tilde{\bm{G}}^+$ through a constant rectangular coefficient matrix $\bm{\alpha} \in\mathbb{R}^{m^-\times m^+}$.
We also introduce a vector $\tilde{\bm{S}}_\mathrm{sur}(t,\bm{x}) : \mathbb{R}^+\times\Gamma \rightarrow \mathbb{R}^l$ of surface sources.
As we will illustrate below, all the physically relevant boundary conditions can indeed be formulated in this way.

Since $\bm{\alpha}$ relates the incoming and outgoing perturbations on the boundary, it can be interpreted as a reflection matrix.
In fact, for a homogeneous boundary condition ($\tilde{\bm{S}}_\mathrm{sur}=\bm{0}$), it can be noted that the boundary is passive if $\|\bm{\alpha} \tilde{\bm{G}}^+\|_2 = \|\tilde{\bm{G}}^-\|_2 \le \|\tilde{\bm{G}}^+\|_2$ for all $\tilde{\bm{G}}^+$.
This holds if and only if the reflection matrix $\bm{\alpha}$ is a contraction:
\begin{align}
	\label{eqn:BChyp}
	\max_{i=1,\dots,m^+} |\lambda_i(\bm{\alpha}^\dagger\bm{\alpha})| \leq 1.
\end{align}

%%%%%%%%%%%%%%%%%%%%%%%%%%%%%%%%%%%%%%%%%%%%%%%%%%

\section{Discontinuous Galerkin method with upwind fluxes}
\label{sec:DG}

In this section, we introduce the discontinuous Galerkin (DG) method that we hybridize.
In particular, we introduce the upwind fluxes in terms of the transmission variables for general time-harmonic problems of the form \eqref{eqn:3DHypSys}.

%%%%%%%%%%%%%%%%%%%%%%%%%%%%%%%%%%%%%%%%%%%%%%%%%%%%%%%%%%%%%%%%%%%%%%%%%%%%%%%%%%%%%%%%

\subsection{General time-harmonic wave propagation problem}\label{sec:general_time_harmonic_problem}

We consider the time-harmonic version of the general system \eqref{eqn:3DHypSys} with boundary conditions of the type \eqref{eqn:BCtime}.
Both the solution and the given data depend implicitly on time as $e^{-\i\omega t}$, where $\i$ denotes the imaginary unit and $\omega$ is the angular frequency.

We are looking for an unknown $l$-dimensional complex-valued vector field $\bm{U}(\bm{x})$ defined on a bounded domain $\Omega\subset\mathbb{R}^3$.
The boundary $\partial\Omega$ is partitioned into $N_\mathrm{sur}>0$ non-overlapping open regions $\{\Gamma_i\}_{i=1,\dots,N_\mathrm{sur}}$ such that $\partial\Omega=\cup_{i=1,\dots,N_\mathrm{sur}}\overline{\Gamma_i}$.
On each region $\Gamma_i$, a fixed number $m_i^-\geq0$ of boundary conditions must be prescribed.
The problem reads as follows.
\begin{problem}
	Find $\bm{U} : \Omega\rightarrow\mathbb{C}^l$ such that
	\begin{align}
		\label{eqn:sysHarm}
		\left\{
		\begin{aligned}
			- \i\omega\bm{U}
			+ \sum_{j=1}^3\bm{A}_j\frac{\partial\bm{U}}{\partial x_j}
			& = \bm{S}_\mathrm{vol},
			&
			& \text{ in } \Omega,
			\\
			\bm{G}^- = \bm{\alpha}_i\bm{G}^+
			& + \bm{S}_{\mathrm{sur},i},
			&
			& \text{ on each } \Gamma_i,
		\end{aligned}
		\right.
	\end{align}
	for given data $\bm{S}_\mathrm{vol} : \Omega\rightarrow\mathbb{C}^l$ and $\bm{S}_{\mathrm{sur},i} : \Gamma_i\rightarrow\mathbb{C}^{m_i}$ with $i=1,\dots,N_\mathrm{sur}$.
	\label{sym_sys}
\end{problem}
\noindent
For each region $\Gamma_i$, the vectors $\bm{G}^+(\bm{x}) \in \mathbb{C}^{m_i^+}$ and $\bm{G}^-(\bm{x}) \in \mathbb{C}^{m_i^-}$ are the time-harmonic versions of the transmission variables \eqref{eqn:transVar}.
The reflection matrices $\bm{\alpha}_i\in\mathbb{R}^{m_i^-\times m_i^+}$ are assumed to satisfy the property \eqref{eqn:BChyp}.
In the frequency domain, i.e. for time-harmonic solutions, the time-averaged energy density and flux are directly given by $\frac{1}{4}\bm{U}^\dagger\bm{U}$ and $\frac{1}{4}\bm{U}^\dagger\bm{F}\bm{U}$, respectively.
Therefore, a boundary is passive whenever $\|\bm{G}^+\|_2\ge\|\bm{G}^-\|_2$ at every point on the boundary.
For a homogeneous boundary condition $\bm{G}^- = \bm{\alpha}_i\bm{G}^+$, this is guaranteed when the reflection matrix satisfies \eqref{eqn:BChyp}.

%%%%%%%%%%%%%%%%%%%%%%%%%%%%%%%%%%%%%%%%%%%%%%%%%%%%%%%%%%%%%%%%%%%%%%%%%%%%%%%%%%%%%%%%

\subsection{Discontinuous Galerkin formulation}
\label{standard_DG_formulation}

Let $\mathcal{T}_h$ be a conforming mesh of the domain $\Omega$ consisting of simplicial elements.
The elements and the faces are denoted by the generic letters $K$ and $F$, respectively.
The collection of element boundaries is denoted by $\partial\mathcal{T}_h := \{\partial K~|~K\in\mathcal{T}_h\}$.
The collection of faces of the mesh is denoted by $\mathcal{F}_h$.
The collection of faces of an element $K$ is denoted by $\mathcal{F}_K$.
For a face $F$ of $K$, $\bm{n}_{K,F}$ is the unit outward normal to $K$ on $F$.

The unknown field is approximated on each element by vectors of polynomials of maximal degree $\mathrm{p}\ge0$, which belongs to
\begin{align}
	\bm{\mathcal{V}}_h^l := \prod_{K\in\mathcal{T}_h}\bm{\mathcal{P}}^l_\mathrm{p}(K),
\end{align}
where $\bm{\mathcal{P}}_\mathrm{p}^l$ denotes the space of $l$-dimensional complex-valued polynomial vectors with polynomial degree smaller or equal to $\mathrm{p}$.
The restriction of a field $\bm{U}_h\in\bm{\mathcal{V}}_h^l$ on $K$ is denoted by $\bm{U}_K$.

We introduce the inner products
\begin{align}
	(\bm{U},\bm{V})_K
	& :=\int_K \bm{V}^\dagger\bm{U} ~ d\bm{x},
	&
	\langle \bm{U},\bm{V}\rangle_{\partial K}
	& :=\sum_{F\in\mathcal{F}_K}\int_F \bm{V}^\dagger\bm{U} ~ d\sigma(\bm{x}),
	\\
	(\bm{U},\bm{V})_{\mathcal{T}_h}
	& :=\sum_{K\in\mathcal{T}_h}(\bm{U},\bm{V})_K,
	&
	\langle \bm{U},\bm{V}\rangle_{\partial\mathcal{T}_h}
	& :=\sum_{K\in\mathcal{T}_h}\langle \bm{U},\bm{V}\rangle_{\partial K}.
\end{align}
In the inner product $\langle\cdot,\cdot\rangle_{\partial K}$, the quantities used in the surface integral correspond to the restriction of fields defined on $K$ (e.g.~$\bm{U}_K$) or quantities associated with the faces of $K$ (e.g.~$\bm{n}_{K,F}$ with $F\in\mathcal{F}_K$), unless otherwise specified.

The DG formulation of Problem \ref{sym_sys} reads as follows.
\begin{problem}[DG formulation]
	\label{pbm:DG_general}
	Find $\bm{U}_h\in\bm{\mathcal{V}}_h^l$ such that, for all $\bm{V}_h\in\bm{\mathcal{V}}_h^l$,
	\begin{align}
		- \imath\omega \big(\bm{U}_h , \bm{V}_h \big)_{\mathcal{T}_h}
		- \sum_{j=1}^3 \Big(\bm{A}_{j}\bm{U}_h,\frac{\partial\bm{V}_h}{\partial x_{j}}\Big)_{\mathcal{T}_h}
		+ \big\langle \widehat{\bm{F}}_h(\bm{U}_h) , \bm{V}_h \big\rangle_{\partial\mathcal{T}_h}
		= 0.
	\end{align}
\end{problem}
\noindent
In this problem, $\widehat{\bm{F}}_h(\bm{U}_h)$ is the \emph{numerical flux} that must be defined.

%%%%%%%%%%%%%%%%%%%%%%%%%%%%%%%%%%%%%%%%%%%%%%%%%%%%%%%%%%%%%%%%%%%%%%%%%%%%%%%%%%%%%%%%

\subsection{Definition of the upwind fluxes}

The numerical flux is an approximation of the physical flux at each mesh face.
It must be defined carefully in order to preserve the consistency and stability of the scheme, see e.g.~\cite{toro2013, leveque2002, quarteroni2013}.
At the interior faces, i.e.~$F\not\subset\partial\Omega$, where the numerical solution $\bm{U}_h$ is a priori discontinuous, it is used to weakly enforce continuity conditions on the physical fields.
At the boundary faces, i.e.~$F\subset\partial\Omega$, it weakly prescribes the boundary conditions.
In this work, upwind fluxes are used.

\subsubsection*{Interior faces}

For an interior face $F\not\subset\partial\Omega$ of an element $K$, the flux matrix associated with the unit outward normal $\bm{n}_{K,F}$ is split as
\begin{align}
	\label{eqn:fluxSplitting}
	\bm{F}_{K,F} = \bm{F}_{K,F}^+ + \bm{F}_{K,F}^-,
	\qquad
	\text{with }
	\bm{F}_{K,F}^\pm := \bm{W}_{K,F}^\pm\bm{\varLambda}_{K,F}^\pm(\bm{W}_{K,F}^\pm)^\dagger,
\end{align}
where $\bm{\varLambda}_{K,F}^\pm$ and $\bm{W}_{K,F}^\pm$ are the eigenvalues and eigenvectors matrices associated to the strictly positive (superscript $^+$)  and strictly negative (superscript $^-$) eigenvalues, as defined in Section \ref{sec:HypSysAndTransVar}.
The \emph{upwind fluxes} are then defined as
\begin{align}
	\label{upwinding_splitting_general}
	\widehat{\bm{F}}_{K,F}(\bm{U}_h)
	:= \bm{F}_{K,F}^+\bm{U}_K + \bm{F}_{K,F}^- \bm{U}_{K'},
\end{align}
where $\bm{U}_K$ is the numerical field associated to the current element and $\bm{U}_{K'}$ is the one associated to the neighboring element sharing the face $F$, see Figure \ref{splitting_fig_general}.
\begin{figure}[!tb]
	\centering
	\includegraphics{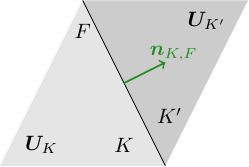}
	\caption[Notation for $\bm{U}_K$ and $\bm{U}_{K'}$ defined on two neighboring elements $K$ and $K'$]{Notation for $\bm{U}_K$ and $\bm{U}_{K'}$ defined on two neighboring elements $K$ and $K'$ that share the face $F$.}
	\label{splitting_fig_general}
\end{figure}
With these numerical fluxes, the numerical scheme is consistent, as the physical fluxes are recovered by replacing the numerical field $\bm{U}_h$ with the exact solution $\bm{U}$ in equation \eqref{upwinding_splitting_general}, i.e.~$\widehat{\bm{F}}_{K,F}(\bm{U}_h) = \widehat{\bm{F}}_{K,F}(\bm{U})$.

\subsubsection*{Reformulation with transmission variables}

The upwind fluxes implicitly mimic the wave propagation by treating information from each side of the interfaces differently.
Similar to \eqref{eqn:transVar}, we define the \textit{outgoing transmission variable} as
\begin{align}
	\bm{G}^{+}_{K,F}
	& := \sqrt{\bm{\varLambda}_{K,F}^+}(\bm{W}_{K,F}^+)^\dagger\bm{U}_{K,F},
	\label{eqn:IFace:transVarOut}
\end{align}
and the \textit{incoming transmission variable} as
\begin{align}
	\bm{G}^{-}_{K,F}
	& := \sqrt{-\bm{\varLambda}_{K,F}^-}(\bm{W}_{K,F}^-)^\dagger\bm{U}_{K',F}.
	\label{eqn:IFace:transVarIn}
\end{align}
By symmetry, the outgoing variable of one element is the incoming variable of the neighboring element, i.e.~$\bm{G}^{+}_{K,F}=\bm{G}^{-}_{K',F}$ and $\bm{G}^{-}_{K,F}=\bm{G}^{+}_{K',F}$.
The upwind fluxes can then be rewritten as
\begin{align}
	\widehat{\bm{F}}_{K,F}(\bm{U}_F)
	= \bm{W}_{K,F}^+\sqrt{ \bm{\varLambda}_{K,F}^+}\bm{G}_{K,F}^{+}
	- \bm{W}_{K,F}^-\sqrt{-\bm{\varLambda}_{K,F}^-}\bm{G}_{K,F}^{-}.
	\label{eqn:IFace:numFlux}
\end{align}
Therefore, quantities corresponding to outward information are computed by using the values of the numerical fields on the current element $K$. Meanwhile, quantities corresponding to incoming information are based on values associated with the neighboring element $K'$.

%%%%%%%%%%%%%%%%%%%%%%%%%%%%%%%%%%%%%%%%%%%%%%%%%%%%%%%%%%%%%%%%%%%%%%%%%%%%%%%%%%%%%%%%

\subsubsection*{Boundary faces}

At boundary faces, $F\subset\partial\Omega$, the numerical fluxes can still be written in terms of outgoing and incoming transmission variables as in equation \eqref{eqn:IFace:numFlux}.
Since the outgoing transmission variable $\bm{G}_{K,F}^{+}$ defined in \eqref{eqn:IFace:transVarOut} depends only on the fields $\bm{U}_K$ in the current element $K$, the definition \eqref{eqn:IFace:transVarOut} of $\bm{G}_{K,F}^{+}$ remain valid for boundary faces.

However, the incoming transmission variable $\bm{G}_{K,F}^-$ used in \eqref{eqn:IFace:numFlux} cannot be written in terms of a solution in a neighboring element.
Instead, the incoming variable $\bm{G}_{K,F}^-$ is directly specified by the generic boundary condition of system \eqref{eqn:sysHarm}.
For a boundary face $F\subset\Gamma_i$ of an element $K$, it is written
\begin{align}
	\bm{G}^{-}_{K,F}
	:= \bm{\alpha}_i \bm{G}^{+}_{K,F} + \bm{S}_{\mathrm{sur},i},
	\qquad
	\text{if } F\subset\Gamma_i.
	\label{bc_transmission}
\end{align}
Therefore, the numerical flux \eqref{eqn:IFace:numFlux} on a boundary face depends solely on $\bm{S}_{\mathrm{sur},i}$ and $\bm{G}_{K,F}^{+}$, and the latter is defined in terms of $\bm{U}_K$ by \eqref{eqn:IFace:transVarOut}.

%%%%%%%%%%%%%%%%%%%%%%%%%%%%%%%%%%%%%%%%%%%%%%%%%%

\section{Hybridization with transmission variables}
\label{sec:CHDG}

Hybridizing a DG formulation consists in introducing additional unknowns on the mesh skeleton in a way that decouples the physical unknowns defined within different elements.
The physical unknowns can then be easily eliminated by solving small, local problems, resulting in a \emph{hybridized system} associated with the additional unknowns.
The physical unknowns can be recovered in a post-processing step after solving this system.

In standard hybridized discontinuous Galerkin (HDG) methods, a numerical trace is generally used as the additional unknown.
In this work, the approach of the CHDG method is extended to a general system \eqref{eqn:sysHarm} by using the incoming transmission variables as additional unknowns on the mesh skeleton.
This change does not affect the final physical solution, but it leads to a hybridized system with improved properties, in particular for iterative solvers.
The hybridization procedure and the hybridized formulation are described in Section \ref{sec:hybForm}, and the key properties are proved in Section \ref{sec:hybProp}.

%%%%%%%%%%%%%%%%%%%%%%%%%%%%%%%%%%%%%%%%%%%%%%%%%%%%%%%%%%%%%%%%%%%%%%%%%%%%%%%%%%%%%%%%

\subsection{Hybridized formulation}
\label{sec:hybForm}

To formulate the problem, we introduce two spaces corresponding to the outgoing and incoming transmission variables on every face of every element:
\begin{align}
	\bm{\mathcal{G}}^+_h :=\prod_{K\in\mathcal{T}_h}\prod_{F\in\mathcal{F}_K}\bm{\mathcal{P}}^{{n}^+_{K,F}}_\mathrm{p}(F)
	\qquad\text{ and }\qquad
	\bm{\mathcal{G}}^-_h := \prod_{K\in\mathcal{T}_h}\prod_{F\in\mathcal{F}_K}\bm{\mathcal{P}}^{{n}^-_{K,F}}_\mathrm{p}(F),
\end{align}
where $n^+_{K,F}$ and $n^-_{K,F}$ are the numbers of strictly positive and strictly negative eigenvalues, respectively, of the flux matrix $\bm{F}_{K,F}$ associated with the normal $\bm{n}_{K,F}$.
These spaces are equipped with the following norm
\begin{align}
	\|\bm{G}_h\|_{\partial\mathcal{T}_h} := \sqrt{\sum_{K\in\mathcal{T}_h}\sum_{F\in\mathcal{F}_K}\|\bm{G}_{K,F}\|^2_F},
\end{align}
where $\|\cdot\|_F:=\sqrt{\langle\cdot,\cdot\rangle_F}$ is the usual $L^2(F)$ norm defined for vectors with the appropriate size.

We introduce the additional unknown field $\bm{G}_h^-\in\bm{\mathcal{G}}^-_h$ corresponding to the incoming variable defined by equation \eqref{eqn:IFace:transVarIn} at each interior face $F\not\in\partial\Omega$ and by equation \eqref{bc_transmission} at each boundary face $F\in\partial\Omega$.
Problem \ref{pbm:DG_general} is augmented with an additional global equation that collects a weak form of these definitions.
The resulting formulation can be written as follows:
\begin{problem}
	\label{pbm:CHDG_general}
	Find $(\bm{U}_h,\bm{G}_h^{-})\in  \bm{\mathcal{V}}_h^l\times\bm{\mathcal{G}}^-_h$ such that, for all $(\bm{V}_h,\bm{\xi}_h)\in \bm{\mathcal{V}}_h^l\times\bm{\mathcal{G}}^-_h$,
	\begin{align}
		- \imath\omega \big(\bm{U}_h,\bm{V}_h\big)_{\mathcal{T}_h}
		- \sum_{j=1}^3 \Big(\bm{A}_{j}\bm{U}_h,\frac{\partial\bm{V}_h}{\partial x_j}\Big)_{\mathcal{T}_h}
		+ \big\langle\bm{F}^+\bm{U}_h - \bm{W}^-\sqrt{-\bm{\varLambda}^-}\bm{G}_h^{-} , \bm{V}_h\big\rangle_{\partial\mathcal{T}_h}
		= 0
	\end{align}
	and
	\begin{align}
		\big\langle \bm{G}^{-}_h-\Pi(\bm{G}^{+}(\bm{U}_h)) , \bm{\xi}_h \big\rangle_{\partial\mathcal{T}_h}
		&= \big\langle \bm{b} , \bm{\xi}_h \big\rangle_{\partial\mathcal{T}_h},
	\end{align}
	with $\bm{G}^{+}(\bm{U}_h) := \sqrt{\bm{\varLambda}^+}(\bm{W}^+)^\dagger\bm{U}_h$.
\end{problem}
\noindent
In the second equation, we have introduced the \emph{global exchange operator} $\Pi: \bm{\mathcal{G}}^+_h\rightarrow\bm{\mathcal{G}}^-_h$ and the \emph{global right-hand side} $\bm{b}$. For each face $F$ of each element $K$, and for all $\bm{G}^{+}\in\bm{\mathcal{G}}^+_h$, they are defined as
\begin{align}
	\Pi(\bm{G}^{+})|_{K,F} =
	\begin{cases}
		\bm{G}^{+}_{K',F}
		& \text{if } F\not\subset\partial\Omega, \\
		\bm{\alpha}_i\bm{G}^{+}_{K,F}
		& \text{if } F\subset\Gamma_i,
	\end{cases}
	\qquad\text{and}\qquad
	\bm{b}|_{K,F} =
	\begin{cases}
		\mathbf{0}
		& \text{if } F\not\subset\partial\Omega, \\
		\bm{S}_{\mathrm{sur},i}
		& \text{if } F\subset\Gamma_i.
	\end{cases}
\end{align}
The operator $\Pi$ enforces the weak coupling of the element-wise problems at the interior faces and the boundary conditions on the boundary faces.
For interior faces, it simply swaps the outgoing and incoming transmission variables of the two neighboring elements.
For boundary faces, its definition is modified to enforce the corresponding boundary conditions.

To obtain the hybridized problem, the physical field $\bm{U}_h$ is eliminated from Problem \ref{pbm:CHDG_general} by solving local element-wise problems where the incoming transmission variable $\bm{G}_h^{-}$ is considered as a given datum.
For each element $K$, the local problem reads as follows.
\begin{problem}[Local problem]
	Find $\bm{U}_K\in\bm{\mathcal{P}}^l_\mathrm{p}(K)$ such that, for all $\bm{V}_K\in\bm{\mathcal{P}}^l_\mathrm{p}(K)$,
	\begin{multline}
		- \imath\omega\big(\bm{U}_K,\bm{V}_K\big)_K
		- \sum_{j=1}^3 \Big(\bm{A}_{j}\bm{U}_K,\frac{\partial\bm{V}_K}{\partial x_j}\Big)_K
		+ \sum_{F\in\mathcal{F}_K} \big\langle\bm{F}_{K,F}^+\bm{U}_K , \bm{V}_K\big\rangle_{F}
		\\
		= \sum_{F\in\mathcal{F}_K} \big\langle\bm{W}_{K,F}^-\sqrt{-\bm{\varLambda}_{K,F}^-}\bm{G}_{K,F}^{-} , \bm{V}_K\big\rangle_{F},
		\label{pbmloc_dg_general}
	\end{multline}
	for a given surface datum $\bm{G}_{K,F}^{-}\in\bm{\mathcal{P}}^l_\mathrm{p}(K)$ for all $F\in\mathcal{F}_K$.
	\label{loc_pbm:chdg_general}
\end{problem}

The hybridized CHDG problem can be written in a convenient abstract form by introducing the \emph{global scattering operator} $\TS:\bm{\mathcal{G}}^-_h\rightarrow \bm{\mathcal{G}}^+_h$ defined such that, for each face $F$ of each element $K$,
\begin{align}
	\TS(\bm{G}^{-}_h)|_{K,F} := \sqrt{\bm{\varLambda}_{K,F}^+}(\bm{W}_{K,F}^+)^\dagger\bm{U}_K,
	\label{scat_oper_general}
\end{align}
where $\bm{U}_K$ is the solution of Problem \ref{loc_pbm:chdg_general} with the variables $\{\bm{G}_{K,F}^{-}\}_{F\in\mathcal{F}_K}$ contained in $\bm{G}_h^{-}$ as a surface data.
This operator can be interpreted as an \emph{``incoming to outgoing''} operator.
By removing the physical unknowns from Problem \ref{pbm:CHDG_general}  and using $\TS$, we obtain the following formulation.
\begin{problem}[Hybridized formulation \texttt{I}]
	\label{pbm:red1_general}
	Find $\bm{G}_h^{-}\in\bm{\mathcal{G}}^-_h$ such that, for all $\bm{\xi}_h\in\bm{\mathcal{G}}^-_h$,
	\begin{align}
		\big\langle \bm{G}_h^{-},\bm{\xi}_h\big\rangle_{\partial\mathcal{T}_h}-\big\langle \Pi(\TS(\bm{G}_h^{-})),\bm{\xi}_h\big\rangle_{\partial\mathcal{T}_h}&=\big\langle \bm{b},\bm{\xi}_h\big\rangle_{\partial\mathcal{T}_h}.
	\end{align}
\end{problem}
\noindent
By introducing the \emph{identity operator} $\TI$ on $\bm{\mathcal{G}}^-_h$ and the \emph{global projected right-hand side vector} $\bm{b}_h\in\bm{\mathcal{G}}^-_h$ defined such that $\langle\bm{b}_h,\bm{\xi}_h\rangle_{\partial\mathcal{T}_h} = \langle\bm{b},\bm{\xi}_h\rangle_{\partial\mathcal{T}_h}$ for all $\bm{\xi}_h\in\bm{\mathcal{G}}^-_h$, it can be rewritten as follows.
\begin{problem}[Hybridized formulation \texttt{II}]
	\label{pbm:red2_general}
	Find $\bm{G}_h^{-}\in\bm{\mathcal{G}}^-_h$ such that
	$(\TI-\Pi\TS)\bm{G}_h^{-} = \bm{b}_h$,
\end{problem}

%%%%%%%%%%%%%%%%%%%%%%%%%%%%%%%%%%%%%%%%%%%%%%%%%%%%%%%%%%%%%%%%%%%%%%%%%%%%%%%%%%%%%%%%

\subsection{Properties of the local and hybridized problems}
\label{sec:hybProp}

In this section, we prove that the local problems (Problem \ref{loc_pbm:chdg_general}) are well-posed, and that the operator $\Pi\TS$ is a strict contraction for the norm $\|\cdot\|_{\partial\mathcal{T}_h}$.
Consequently, Problems \ref{pbm:red1_general} and \ref{pbm:red2_general} are also well-posed and equivalent to Problem \ref{pbm:CHDG_general}.

\begin{theorem} [Well-posedness of the local problem]
	Problem \ref{loc_pbm:chdg_general} is well-posed.
	\label{well_posedness_loc_pbm_general}
\end{theorem}
\begin{proof}
	We have to prove that, if $\bm{G}_{K,F}^{-}=\mathbf{0}$ on each face $F\in\mathcal{F}_K$, the unique solution of Problem \ref{loc_pbm:chdg_general} is $\bm{U}_K=\bm{0}$.
	For the sake of brevity, the subscripts $K$ and $F$ are omitted for the local fields, the coefficient matrices and the outgoing unit normal. Taking $\bm{V}=\bm{U}$ in \eqref{pbmloc_dg_general} gives
	\begin{align}
		-\imath\omega(\bm{U},\bm{U})_K
		- \sum_{j=1}^3 \Big(\bm{A}_{j}\bm{U},\frac{\partial\bm{U}}{\partial x_j}\Big)_K
		+ \langle\bm{F}^+\bm{U},\bm{U}\rangle_{\partial K} = 0.
	\end{align}
	Integrating by parts the second term, taking the complex conjugate of the resulting equation, recalling $\sum_j n_j\bm{A}_j = \bm{F} = \bm{F}^++\bm{F}^-$, and using the symmetry of these matrices yields
	\begin{align}
		+\imath\omega(\bm{U},\bm{U})_K+\Big(\bm{A}_{j}\bm{U},\frac{\partial\bm{U}}{\partial x_j}\Big)_K-\langle\bm{F}^-\bm{U},\bm{U}\rangle_{\partial K}=0.
	\end{align}
	Summing the two previous equations gives $\langle (\bm{F}^+-\bm{F}^-)\bm{U},\bm{U}\rangle_{\partial K} = 0$, which can be written
	\begin{equation}
		\big\|\sqrt{\bm{\varLambda}^+}(\bm{W}^+)^\dagger\bm{U}\big\|_{\partial K}^2
		+ \big\|\sqrt{-\bm{\varLambda}^-}(\bm{W}^-)^\dagger\bm{U}\big\|_{\partial K}^2 = 0.
	\end{equation}
	Therefore, we have $\sqrt{\pm\bm{\varLambda}^\pm}(\bm{W}^\pm)^\dagger\bm{U} = \mathbf{0}$, which implies that $\bm{F}^+\bm{U} = \mathbf{0}$ and $\bm{F}^-\bm{U} = \mathbf{0}$ on each face $F\in\mathcal{F}_K$.
	Using this result in Problem \ref{loc_pbm:chdg_general}, $\bm{U}$ must satisfy, for all $\bm{V}\in\bm{\mathcal{P}}_{\mathrm{p}}^l(K)$,
	\begin{align}
		-\i\omega(\bm{U},\bm{V})_K
		- \sum_{j=1}^3\Big(\bm{A}_j\bm{U},\frac{\partial\bm{V}}{\partial x_j}\Big)_K = 0.
	\end{align}
	Therefore, $\bm{U}$ is a solution of the strong form of the problem.
	Since there are no non-zero polynomial solutions to this problem, we have $\bm{U} = \mathbf{0}$, which yields the result.
\end{proof}

The strict contraction of $\Pi\TS$ is a consequence of the strict contraction of $\TS$ and the contraction of $\Pi$.
To prove this result, we need the following lemma.
\begin{lemma}
	(i) The solution of Problem \ref{loc_pbm:chdg_general} verifies
	\begin{multline}
		\sum_{F\in\mathcal{F}_K} \big\|\sqrt{\bm{\varLambda}_{K,F}^+}(\bm{W}_{K,F}^+)^\dagger\bm{U}_{K}\big\|_F^2
		+ \sum_{F\in\mathcal{F}_K} \big\|\sqrt{-\bm{\varLambda}_{K,F}^-}(\bm{W}_{K,F}^{-})^\dagger\bm{U}_{K}-\bm{G}_{K,F}^{-}\big\|_F^2
		\\
		= \sum_{F\in\mathcal{F}_K} \big\|\bm{G}_{K,F}^{-}\big\|_F^2.
		\label{eqn:lem2_general}
	\end{multline}
	(ii) The second term on the left-hand side of \eqref{eqn:lem2_general} vanishes if and only if $\bm{G}_{K,F}^{-}=\mathbf{0}$.
	\label{well-pos-loc2_general}
\end{lemma}
\begin{proof}
	For brevity, the subscripts $K$ and $F$ are omitted for the local fields, the coefficient matrices, the unit outgoing normal and the surface data.
	
	\emph{(i)} We proceed as in the proof of Theorem \ref{well_posedness_loc_pbm_general}.
	Taking \eqref{pbmloc_dg_general} with $\bm{V}=\bm{U}$ and summing it with the equation obtained by using an integration by parts and the complex conjugate of \eqref{pbmloc_dg_general} gives
	\begin{align}
		\sum_{F\in\mathcal{F}_K} \langle(\bm{F}^+-\bm{F}^-)\bm{U},\bm{U}\rangle_F
		= \sum_{F\in\mathcal{F}_K} \langle\bm{W}^{-}\sqrt{-\bm{\varLambda}^-}\bm{G}^{-},\bm{U}\rangle_F
		+ \sum_{F\in\mathcal{F}_K}\langle\bm{U},\bm{W}^{-}\sqrt{-\bm{\varLambda}^-}\bm{G}^{-}\rangle_F.
	\end{align}
	Using the definition of $\bm{F}^+$ and $\bm{F}^-$ leads to
	\begin{multline}
		\sum_{F\in\mathcal{F}_K} \left(
		\big\langle\bm{W}^+\bm{\Lambda}^+(\bm{W}^+)^\dagger\bm{U},\bm{U}\big\rangle_F -
		\big\langle\bm{W}^-\bm{\Lambda}^-(\bm{W}^-)^\dagger\bm{U},\bm{U}\big\rangle_F\right)
		\\
		=
		\sum_{F\in\mathcal{F}_K} \langle\bm{G}^{-},\sqrt{-\bm{\Lambda}^-}(\bm{W}^-)^\dagger\bm{U}\rangle_F
		+ \sum_{F\in\mathcal{F}_K} \langle\sqrt{-\bm{\Lambda}^-}(\bm{W}^-)^\dagger\bm{U},\bm{G}^{-}\rangle_F,
	\end{multline}
	and then
	\begin{align}
		\sum_{F\in\mathcal{F}_K} \left(
		\big\|\sqrt{\bm{\varLambda}^+}(\bm{W}^+)^\dagger\bm{U}\big\|^2_F
		+ \big\langle\sqrt{-\bm{\varLambda}^-}(\bm{W}^-)^\dagger\bm{U}-\bm{G}^{-},\sqrt{-\bm{\varLambda}^-}(\bm{W}^-)^\dagger\bm{U}-\bm{G}^{-}\big\rangle_F
		- \big\|\bm{G}^{-}\big\|^2_F
		\right) = 0,
	\end{align}
	which gives the result.
	
	\emph{(ii)} The second term on the left-hand side of \eqref{eqn:lem2_general} vanishes if and only if $\bm{G}^{-}=\sqrt{-\bm{\varLambda}^-}(\bm{W}^-)^\dagger\bm{U}$ on each face $F\in\mathcal{F}_K$.
	Using this relation in Problem \ref{loc_pbm:chdg_general}, $\bm{U}$ must satisfy, for all $\bm{V}\in\bm{\mathcal{P}}^{l}_\mathrm{p}(K)$,
	\begin{align}
		-\i\omega(\bm{U},\bm{V})_K
		- \sum_{j=1}^3\Big(\bm{A}_{j}\bm{U},\frac{\partial\bm{V}}{\partial x_j}\Big)_K
		+ \langle\bm{F}\bm{U},\bm{V}\rangle_{\partial K} = 0.
	\end{align}
	Therefore, $\bm{U}$ solves the system in strong form.
	Because there is no non-zero polynomial solution to the previous equation, we have $\bm{U}=\mathbf{0}$, and then $\bm{G}^{-}=\mathbf{0}$, which yields the results.
	The converse statement is direct, because the local problem is well-posed.
\end{proof}

\begin{theorem}
	\label{thm:opS:cont_general}
	The scattering operator $\TS$ is a strict contraction, i.e.
	\begin{align}
		\|\TS(\bm{G}_h^{-})\|_{\partial\mathcal{T}_h} < \|\bm{G}_h^{-}\|_{\partial\mathcal{T}_h}, \quad \forall \bm{G}_h^{-}\in\bm{\mathcal{G}}^-_h\backslash\{\mathbf{0}\}.
	\end{align}
\end{theorem}
\begin{proof}
	Let $\bm{G}_h^{-}\in\bm{\mathcal{G}}^-_h\backslash\{\mathbf{0}\}$.
	By Lemma \ref{well-pos-loc2_general}, we have
	\begin{align}
		\sum_{F\in\mathcal{F}_K}\|\sqrt{\bm{\varLambda}^+}(\bm{W}^+)^\dagger\bm{U}\|_F^2&\le\sum_{F\in\mathcal{F}_K}\|\bm{G}_{K,F}^{-}\|_F^2.
	\end{align}
	The equality holds if and only if $\bm{G}_{K,F}^{-}=\mathbf{0}$ for all $F\in\mathcal{F}_K$.
	By using the definition \eqref{scat_oper_general} of $\TS$, we can write
	\begin{align}
		\sum_{F\in\mathcal{F}_K}\|\TS(\bm{G}_h^{-})|_{K,F}\|^2_F\le\sum_{F\in\mathcal{F}_K}\|\bm{G}_{K,F}^{-}\|^2_F.
	\end{align}
	Summing this estimate over all $K\in\mathcal{T}_h$ and noticing that the equality can not globally hold because $\bm{G}_h^{-}\ne\mathbf{0}$ give the result.
\end{proof}

\begin{theorem}
	\label{thm:opP:cont_general}
	The exchange operator $\Pi$ is a contraction, i.e.
	\begin{align}
		\|\Pi(\bm{G}^{+})\|_{\partial\mathcal{T}_h} \le \|\bm{G}^{+}\|_{\partial\mathcal{T}_h}, \quad \forall \bm{G}^{+}\in\bm{\mathcal{G}}^+_h.
	\end{align}
\end{theorem}
\begin{proof}
	This is a direct consequence of the definition of $\Pi$ and of the passivity of the boundary, i.e.~$\bm{\alpha}_i$ satisfies Equation \eqref{eqn:BChyp} in each region $\Gamma_i$.
\end{proof}

\begin{corollary}
	\label{corol_general}
	The operator $\Pi\TS$ is a strict contraction, i.e.
	\begin{align}
		\|\Pi\TS(\bm{G}^{-}_h)\|_{\partial\mathcal{T}_h} < \|\bm{G}^{-}_h\|_{\partial\mathcal{T}_h}, \quad \forall \bm{G}^{-}_h\in\bm{\mathcal{G}}^-_h\backslash\{\mathbf{0}\}.
	\end{align}
\end{corollary}
\begin{proof}
	This is a direct consequence of Theorems \ref{thm:opS:cont_general} and \ref{thm:opP:cont_general}.
\end{proof}

The strict contraction of $\TS$ can be interpreted as a numerical effect of the scheme, since this operator would be an isometry for the continuous problem.
In contrast, the contraction of the exchange operator $\Pi$ is related to the conservation of energy in the physical problem.
From \eqref{eqn:consEnergyFlux} we can see that the contraction of $\Pi$ implies that the outgoing energy is larger than, or equal to, the incoming energy.
In other words, the boundary conditions are passive, and the permutation operation $\Pi$ cannot add energy to the system.
It is an isometry if the boundary conditions preserve energy, i.e. when they do not absorb or add energy to the system.

%%%%%%%%%%%%%%%%%%%%%%%%%%%%%%%%%%%%%%%%%%%%%%%%%%

\section{Application to aeroacoustics}
\label{sec:aeroac}

In this section, the method is applied to an aeroacoustic problem, namely the propagation of linear waves in a uniform mean flow.
After introducing the propagation model, the equations are written in matrix form to fit into the general framework introduced in the previous sections.
Then, the upwind fluxes, the transmission variables et the operators used to define the CHDG hybridized system are derived.

%%%%%%%%%%%%%%%%%%%%%%%%%%%%%%%%%%%%%%%%%%%%%%%%%%%%%%%%%%%%%%%%%%%%%%%%%%%%%%%%%%%%%%%%

\subsection{Time-harmonic wave propagation in a uniform mean flow}

We consider the propagation of linear waves in a bounded domain $\Omega\subset\mathbb{R}^3$ with a constant density $\rho_0$, a constant sound speed $c_0$ and a uniform mean flow velocity $\bm{u}_0$.
The pressure perturbation $p(\bm{x})$ and the velocity perturbation $\bm{u}(\bm{x})$ are governed by the time-harmonic version of the linearized Euler equations (LEE).
We assume that the boundary of the domain is passive, and that the mean flow is subsonic (i.e.~$|\bm{u}_0| < c_0$).

On the boundary $\partial\Omega$, we define an orthonormal basis composed of the outward unit normal vector $\bm{n}$ and two unit tangential vectors $\{\bm{\tau}_1,\bm{\tau}_2\}$, as illustrated in Figure \ref{fig:aero:localFrame}.
We also introduce the tangent matrix $\bm{T} := \begin{bmatrix}\bm{\tau}_1 & \bm{\tau}_2\end{bmatrix}$ and the projection operator $\bm{\pi} := \bm{I}_3-\bm{n}\bm{n}^\dagger = \bm{TT}^\dagger$, see Figure \ref{fig:aero:projection}.
\begin{figure}[!b]
	\centering
	\begin{subfigure}[t]{0.3\textwidth}
		\centering
		\includegraphics{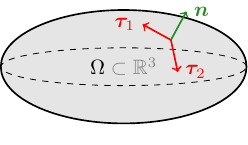}
		\caption{Normal and tangential unit vectors.}
		\label{fig:aero:localFrame}
	\end{subfigure}
	\hfill
	\begin{subfigure}[t]{0.3\textwidth}
		\centering
		\includegraphics{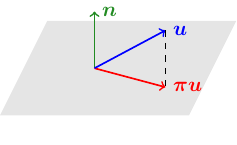}
		\caption{Projection $\bm{\pi}\bm{u}$ of a vector $\bm{u}$ on a plane surface.}
		\label{fig:aero:projection}
	\end{subfigure}
	\hfill
	\begin{subfigure}[t]{0.3\textwidth}
		\centering
		\includegraphics{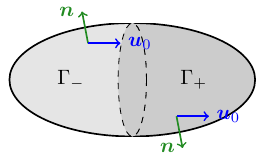}
		\caption{Outflow and inflow boundary.}
		\label{fig:aero:boundaries}
	\end{subfigure}
	\caption{Notation on the boundary (a)-(c) and on the projection operator (b).}
\end{figure}
The number of boundary conditions depends on the number of negative eigenvalues of the flux matrix.
As shown below, the numbers of boundary conditions will therefore depend on the sign of $\bm{u}_0\cdot\bm{n}$.
For this reason, we define two different partitions of the boundary $\partial\Omega$ into non-overlapping regions.
The first partition, $\Gamma_{p}\cup\Gamma_{u}\cup\Gamma_{\mathrm{imp}}$, is related to the boundary conditions enforcing either pressure, normal velocity or impedance.
The second partition, $\Gamma_+\cup\Gamma_-$, separates the region $\Gamma_+$ with outgoing or null mean flow (i.e.~$\bm{u}_0\cdot\bm{n}\geq0$) from the region $\Gamma_-$ with incoming mean flow (i.e.~$\bm{u}_0\cdot\bm{n}<0$), as shown in Figure \ref{fig:aero:boundaries}.

The problem reads as follows.
\begin{problem}\label{prb:aero}
	Find $(p,\bm{u}):\Omega\rightarrow \mathbb{C}\times\mathbb{C}^3$ such that
	\begin{align}
		\left\{
		\begin{aligned}
			-\i\omega p + \bm{u}_0\cdot\grad p + \rho_0c_0^2 \div\bm{u}
			& = s_\mathrm{vol}
			&
			& \text{in } \Omega,
			\\
			-\i\omega\rho_0c_0\bm{u} + c_0\grad p + \rho_0c_0(\bm{u}_0\cdot\grad)\bm{u}
			& = \mathbf{0}_3
			&
			& \text{in } \Omega,
			\\
			p
			& = s_{p}
			&
			& \text{on } \Gamma_{p},
			\\
			\bm{u}\cdot\bm{n}
			& = s_u
			&
			& \text{on } \Gamma_{u},
			\\
			p-\rho_0c_0\bm{u}\cdot\bm{n}
			& = s_{\mathrm{imp}}
			&
			& \text{on } \Gamma_{\mathrm{imp}},
			\\
			\bm{T}^\dagger\bm{u}
			& = \bm{s}_-
			&
			& \text{on } \Gamma_-,
		\end{aligned}
		\right.
	\end{align}
	for given data
	$s_\mathrm{vol}:\Omega\rightarrow\mathbb{C}$,
	$s_{p}: \Gamma_{p} \to \mathbb{C}$,
	$s_u: \Gamma_{u} \to \mathbb{C}$,
	$s_{\mathrm{imp}}: \Gamma_{\mathrm{imp}} \to \mathbb{C}$ and
	$\bm{s}_-: \Gamma_- \to \mathbb{C}^{2}$.
\end{problem}

The first two equations in this problem are the linearised conservation of mass and momentum, respectively.
The first three boundary conditions prescribe the pressure, the normal velocity and the impedance.
The last boundary condition is only prescribed on the inflow region $\Gamma_-$ and specifies the tangential velocity.
The use of this last boundary condition will be discussed in the next section.

Furthermore, it is important to note that this propagation model supports both sound and vorticity waves, which are markedly different in nature, see Section~5.1 in \cite{tam2012computational}.
Sound waves induce only a curl-free velocity field while vorticity waves induce only a divergence-free velocity field.
In addition, the presence of the mean flow changes the phase velocity of the sound waves depending on their direction of propagation (one is therefore dealing with an anisotropic medium).

%%%%%%%%%%%%%%%%%%%%%%%%%%%%%%%%%%%%%%%%%%%%%%%%%%%%%%%%%%%%%%%%%%%%%%%%%%%%%%%%%%%%%%%%

\subsection{Matrix form of the problem and transmission variables}

%%%%%%%%%%%%%%%%%%%%%%%%%%%%%%%%%%%%%%%%%%%%%%%%%%%%%%%%%%%%%%%%%%%%%%%%%%%%%%%%%%%%%%%%

\subsubsection*{Governing equations and conservation of energy}

Problem \ref{prb:aero} can be easily expressed in the form of Problem \ref{sym_sys}.
The \emph{unknown vector}, the \emph{coefficient matrices} and the \emph{volume source vector} are given by
\begin{align}
	\bm{U} =
	\begin{bmatrix} p \\ \rho_0c_0\bm{u} \end{bmatrix},
	\quad
	\bm{A}_{j} =
	\begin{bmatrix}
		\bm{e}_j\cdot\bm{u}_0 & c_0 \bm{e}^\dagger_j            \\
		c_0 \bm{e}_j          & (\bm{e}_j\cdot\bm{u}_0)\bm{I}_3 \\
	\end{bmatrix}
	\quad \text{ and } \quad
	\bm{S}_\mathrm{vol} =
	\begin{bmatrix} s_\mathrm{vol} \\ \mathbf{0}_3 \end{bmatrix},
\end{align}
with the Cartesian direction vectors $\{\bm{e}_j\}_{j=1\dots 3}\in\mathbb{R}^3$.
The scaling of the velocity $\bm{u}$ by $\rho_0c_0$ is introduced to obtain symmetric matrices $\bm{A}_j$.
On a boundary with unit outward normal $\bm{n}$, the \emph{flux matrix} and the \emph{physical fluxes} read
\begin{align}
	\label{eqn:aero:fluxMatrix}
	\bm{F} =
	\begin{bmatrix}
		\bm{u}_0\cdot\bm{n} & c_0\bm{n}^\dagger           \\
		c_0\bm{n}           & (\bm{u}_0\cdot\bm{n})\bm{I}_3
	\end{bmatrix}
	\qquad\text{and}\qquad
	\bm{F}\bm{U} =
	\begin{bmatrix}
		(\bm{u}_0\cdot\bm{n}) p + \rho_0c_0^2 (\bm{u}\cdot\bm{n}) \\
		c_0 \bm{n} \: p + \rho_0c_0 (\bm{u}_0\cdot\bm{n}) \bm{u}
	\end{bmatrix}.
\end{align}
The boundary conditions prescribed at the boundary of the domain are written in matrix form below, after the definition of the transmission variables.

As discussed in Section~\ref{sec:general_time_harmonic_problem}, the energy conservation equation \eqref{eqn:consEnergy} also applies to time-harmonic solutions.
The time-averaged values of the \emph{energy density} $e$ and the \emph{energy flux} $f$ are defined as
\begin{equation}
	\label{eq:aero:energy}
	e
	:= \frac{1}{4\rho_0c_0^2} \bm{U}^\dagger\bm{U}
	= \frac{|p|^2}{4\rho_0 c_0^2}
	+ \frac{1}{4} \rho_0 |\bm{u}|^2
	\quad\text{and}\quad
	f
	:= \frac{1}{4\rho_0c_0^2} \bm{U}^\dagger\bm{F}\bm{U}
	= e (\bm{u}_0\cdot\bm{n}) + \tfrac{1}{2} \Re\{p\bm{u}^\dagger\bm{n}\},
\end{equation}
respectively.
We have divided by $\rho_0c_0^2$ to obtain units consistent with an energy density and a flux.
Note that the energy defined above is the perturbation energy, and not the acoustic energy.
It contains the energy carried by both the sound waves and the vorticity waves.
The well-known definitions of the energy density and flux for sound waves alone can be found in Section 1.7 in \cite{goldstein76}, but they differ from \eqref{eq:aero:energy}.
The definition of perturbation energy is discussed in detail in \cite{myersExactEnergyCorollary1986,myers91}.
It is possible to show that Equation \eqref{eq:aero:energy} is consistent with Equation (20) in \cite{myersExactEnergyCorollary1986} when assuming a uniform mean flow.

%%%%%%%%%%%%%%%%%%%%%%%%%%%%%%%%%%%%%%%%%%%%%%%%%%%%%%%%%%%%%%%%%%%%%%%%%%%%%%%%%%%%%%%%

\subsubsection*{Characteristic analysis and transmission variables}

To perform the characteristic analysis, we consider the flux matrix \eqref{eqn:aero:fluxMatrix} for a  given unit vector $\bm{n}$.
Its eigenvalues are
\begin{align}
	\lambda_1 = \bm{u}_0\cdot\bm{n}+c_0>0,
	\quad
	\lambda_2 = \bm{u}_0\cdot\bm{n}-c_0<0
	\quad\text{and}\quad
	\lambda_3=\lambda_4= \bm{u}_0\cdot\bm{n},
\end{align}
and a basis of associated unit eigenvectors is given by
\begin{align}
	\bm{w}_1 = \frac{1}{\sqrt{2}}\begin{bmatrix} 1 \\ \bm{n}\end{bmatrix},
	\quad
	\bm{w}_2 = \frac{1}{\sqrt{2}}\begin{bmatrix} 1 \\ -\bm{n}\end{bmatrix},
	\quad
	\bm{w}_{3} = \begin{bmatrix} 0 \\ \bm{\tau}_1\end{bmatrix}
	\quad\text{and}\quad
	\bm{w}_4 = \begin{bmatrix} 0 \\ \bm{\tau}_2\end{bmatrix}.
\end{align}
Since the mean flow is subsonic, we have $|\bm{u}_0\cdot\bm{n}|<c_0$.
The first two eigenvalues are then always strictly positive and strictly negative, respectively.
They represent sound waves propagating in the direction $\bm{n}$ and in the opposite direction, respectively.
The sign of the two last eigenvalues $\bm{u}_0\cdot\bm{n}$ depends solely on the direction of the mean flow relative to $\bm{n}$.
These two eigenvalues represent vorticity waves convected by the mean flow in the direction $\bm{n}$ if $\bm{u}_0\cdot\bm{n}>0$ and in the opposite direction if $\bm{u}_0\cdot\bm{n}<0$.
As mentioned above, these set of characteristics therefore shows that the linearised Euler equations support both sound and vorticity waves.

Due to the associated eigenvectors, the first two eigenvalues correspond to transmission variables expressed as $p$ and $\bm{u}\cdot\bm{n}$, while the remaining two correspond to transmission variables expressed as $\bm{T}^\dagger\bm{u}$.
The transmission variables are therefore referred to as normal and tangential, respectively.
Using definition \eqref{eqn:transVar}, they can be written as
\begin{align}
	\bm{G}^{+}
	=
	\begin{cases}
		\begin{bmatrix}
			g^{+,\mathrm{n}} \\
			\bm{g}^{+,\mathrm{t}}
		\end{bmatrix}
		& \text{if $\bm{u}_0\cdot\bm{n}>0$,}
		\vspace{1mm}
		\\
		\begin{bmatrix}
			g^{+,\mathrm{n}}
		\end{bmatrix}
		& \text{otherwise,}
	\end{cases}
	\qquad\text{and}\qquad
	\bm{G}^{-}
	=
	\begin{cases}
		\begin{bmatrix}
			g^{-,\mathrm{n}} \\
			\bm{g}^{-,\mathrm{t}}
		\end{bmatrix}
		& \text{if $\bm{u}_0\cdot\bm{n}<0$,}
		\vspace{1mm}
		\\
		\begin{bmatrix}
			g^{-,\mathrm{n}}
		\end{bmatrix}
		& \text{otherwise,}
	\end{cases}
\end{align}
with the normal (superscript $^\mathrm{n}$) and tangential (superscript $^\mathrm{t}$) variables defined as
\begin{align}
	g^{+,\mathrm{n}}
	& := \sqrt{\lambda_1}\bm{w}_1^\dagger\bm{U}
	= \sqrt{c_0^+/2} \ (p+\rho_0c_0\bm{u}\cdot\bm{n}),
	\label{eqn:aero:gPlusN}
	\\
	g^{-,\mathrm{n}}
	& := \sqrt{-\lambda_2}\bm{w}_2^\dagger\bm{U}
	= \sqrt{c_0^-/2} \ (p-\rho_0c_0\bm{u}\cdot\bm{n}),
	\\
	\bm{g}^{+,\mathrm{t}}
	& := \sqrt{\lambda_{3,4}}\begin{bmatrix}\bm{w}_3 & \bm{w}_4\end{bmatrix}^\dagger\bm{U}
	= \rho_0c_0\sqrt{u_0^\mathrm{n}} \ \bm{T}^\dagger\bm{u},
	&
	& \text{(if $u_0^\mathrm{n}>0$)}
	\label{eqn:aero:gPlusT}
	\\
	\bm{g}^{-,\mathrm{t}}
	& := \sqrt{-\lambda_{3,4}}\begin{bmatrix}\bm{w}_3 & \bm{w}_4\end{bmatrix}^\dagger\bm{U}
	= \rho_0c_0\sqrt{-u_0^\mathrm{n}} \ \bm{T}^\dagger\bm{u}.
	&
	& \text{(if $u_0^\mathrm{n}<0$)}
\end{align}
with the short-hand notations $u_0^\mathrm{n} := \bm{u}_0\cdot\bm{n}$ and $c_{0}^{\pm} := c_0 \pm u_0^\mathrm{n}$.

%%%%%%%%%%%%%%%%%%%%%%%%%%%%%%%%%%%%%%%%%%%%%%%%%%%%%%%%%%%%%%%%%%%%%%%%%%%%%%%%%%%%%%%%

\subsubsection*{Boundary conditions and passivity condition}

On the boundary $\partial\Omega$ of the domain, the different boundary conditions introduced in Problem \ref{prb:aero} can be written in terms of the transmission variables as follows:
\begin{align}
	g^{+,\mathrm{n}} \big/ \sqrt{c_0^+}
	+ g^{-,\mathrm{n}} \big/ \sqrt{c_0^-}
	& = \sqrt{2} s_{p}
	&
	& \text{on } \Gamma_{p},
	\\
	g^{+,\mathrm{n}} \big/ \sqrt{c_0^+}
	- g^{-,\mathrm{n}} \big/ \sqrt{c_0^-}
	& = \sqrt{2}\rho_0 c_0 s_u
	&
	& \text{on } \Gamma_{u},
	\\
	g^{-,\mathrm{n}}
	& = \sqrt{c_0^-/2} \ s_{\mathrm{imp}}
	&
	& \text{on } \Gamma_{\mathrm{imp}},
	\\
	\bm{g}^{-,\mathrm{t}}
	& = \rho_0c_0\sqrt{-u_0^\mathrm{n}} \ \bm{s}_-
	&
	& \text{on } \Gamma_-. \qquad \text{(if $u_0^\mathrm{n}<0$)}
\end{align}
For each region of the boundary, these conditions can be written in matrix form $\bm{G}^{-} - \bm{\alpha} \bm{G}^{+} = \bm{S}_\mathrm{sur}$ with the boundary matrix $\bm{\alpha}$ defined as
\begin{align}
	\bm{\alpha}
	=
	\begin{cases}
		\begin{bmatrix} \alpha & \bm{0}^\dagger \end{bmatrix} & \text{if $u_0^\mathrm{n}>0$,} \vspace{1mm} \\
		\begin{bmatrix} \alpha \end{bmatrix}                  & \text{if $u_0^\mathrm{n}=0$,} \vspace{1mm} \\
		\begin{bmatrix} \alpha \\ \bm{0} \end{bmatrix}        & \text{if $u_0^\mathrm{n}<0$,}
	\end{cases}
	\qquad\text{with }
	\alpha
	:= \begin{cases}
		-\sqrt{c_0^-/c_0^+} & \text{on } \Gamma_{p},            \\
		\sqrt{c_0^-/c_0^+}  & \text{on } \Gamma_{u},            \\
		0                                                           & \text{on } \Gamma_{\mathrm{imp}}, \\
	\end{cases}
\end{align}
and the surface source vector $\bm{S}_\mathrm{sur}$ defined as
\begin{align}
	\bm{S}_\mathrm{sur}
	=
	\begin{cases}
		\begin{bmatrix} g_\mathrm{sur}^{\mathrm{n}} \end{bmatrix}                                                  & \text{if $u_0^\mathrm{n}\geq0$,} \vspace{1mm} \\
		\begin{bmatrix} g_\mathrm{sur}^{\mathrm{n}} \\ \bm{g}_\mathrm{sur}^{\mathrm{t}} \end{bmatrix} & \text{if $u_0^\mathrm{n}<0$,}
	\end{cases}
	\qquad\text{with }
	g_\mathrm{sur}^{\mathrm{n}}
	:= \begin{cases}
		\sqrt{2c_0^+} \: s_p             & \text{on } \Gamma_{p},            \\
		-\sqrt{2c_0^+} \: \rho_0c_0s_u   & \text{on } \Gamma_{u},            \\
		\sqrt{c_0^-/2} \: s_\mathrm{imp} & \text{on } \Gamma_{\mathrm{imp}},
	\end{cases}
\end{align}
and $\bm{g}_\mathrm{sur}^{\mathrm{t}} := \rho_0c_0\sqrt{-u_0^\mathrm{n}} \ \bm{s}_-$.

The passivity condition in Equation \eqref{eqn:BChyp} holds if and only if $|\alpha| \leq 1$.
It is always the case on $\Gamma_\mathrm{imp}$, but it is the case on $\Gamma_p$ and $\Gamma_{u}$ if and only if $u_0^\mathrm{n} \geq 0$.
Therefore, the passivity condition requires the following hypothesis.

\begin{assumption}[Passivity condition]
	\label{hyp:aero:passivity}
	The background flow is grazing or outgoing, i.e.~$u_0^\mathrm{n}\ge0$, on $\Gamma_{p}\cup\Gamma_{u}$, or equivalently, $\Gamma_{p}\cup\Gamma_{u}\subset\Gamma_+$.
\end{assumption}

%%%%%%%%%%%%%%%%%%%%%%%%%%%%%%%%%%%%%%%%%%%%%%%%%%%%%%%%%%%%%%%%%%%%%%%%%%%%%%%%%%%%%%%%

\subsection{Discontinuous Galerkin method with upwind fluxes}

The DG formulation is given by Problem \ref{pbm:DG_general} with $\bm{\mathcal{V}}_h^l:=\bm{\mathcal{V}}_h^1\times\bm{\mathcal{V}}_h^3$.
For the sake of brevity, in the following, we denote $\mathcal{V}_h=\bm{\mathcal{V}}_h^1$ and $\bm{\mathcal{V}}_h=\bm{\mathcal{V}}_h^3$.
For the sake of simplicity, we also set $s_\mathrm{vol}=0$.

It can be written in terms of the physical unknowns as follows
\begin{problem}
	Find $(p_h,\bm{u}_h)\in \mathcal{V}_h\times\bm{\mathcal{V}}_h$ such that, for all $(q_h,\bm{v}_h)\in \mathcal{V}_h\times\bm{\mathcal{V}}_h$,
	\begin{align}
		\left\{
		\begin{aligned}
			& - \i\omega\big(p_h, q_h\big)_{\mathcal{T}_h}
			- \big(p_h,\bm{u}_0\cdot\grad q_h\big)_{\mathcal{T}_h}
			- \big(\rho_0c_0^2\bm{u}_h, \grad q_h\big)_{\mathcal{T}_h}
			\\
			& \qquad\qquad\qquad\qquad\qquad\qquad\qquad\qquad
			+ \big\langle(\bm{u}_0\cdot\bm{n}) \widehat{p}_h + \rho_0c_0^2(\bm{n}\cdot\widehat{\bm{u}}_h), q_h\big\rangle_{\partial\mathcal{T}_h}
			= 0,
			\\
			& - \i\omega\big(\rho_0c_0\bm{u}_h, \bm{v}_h\big)_{\mathcal{T}_h}
			- \big(\rho_0c_0\bm{u}_h,\grad\bm{v}_h\cdot\bm{u}_0\big)_{\mathcal{T}_h}
			- \big(c_0p_h, \grad\cdot\bm{v}_h\big)_{\mathcal{T}_h}
			\\
			& \qquad\qquad\qquad\qquad\qquad\qquad\qquad\qquad
			+ \big\langle c_0\widehat{p}_h\bm{n}+\rho_0c_0(\bm{u}_0\cdot\bm{n})\widehat{\bm{u}}_h,\bm{v}_h\big\rangle_{\partial\mathcal{T}_h}
			= 0,
		\end{aligned}
		\right.
	\end{align}
	with the numerical fluxes defined below.
\end{problem}

For each face $F$ of each element $K$, the numerical fluxes can be written in terms of the transmission variables by using equation \eqref{eqn:IFace:numFlux} with the matrices resulting from the flux splitting
\begin{align}
	\bm{\varLambda}^+_{K,F}
	& =
	\begin{cases}
		\begin{bmatrix}
			c_0^+ & 0                       \\
			0     & u_0^\mathrm{n} \bm{I}_2
		\end{bmatrix}
		& \text{if } u_0^\mathrm{n}>0,
		\vspace{1mm}
		\\
		\begin{bmatrix}
			c_0^+
		\end{bmatrix}
		& \text{otherwise},
	\end{cases}
	&
	\bm{W}^+_{K,F}
	& =
	\begin{cases}
		\begin{bmatrix}
			\frac{1}{\sqrt{2}}             & \bm{0}_2^\dagger \\
			\frac{1}{\sqrt{2}}\bm{n}_{K,F} & \bm{T}_{K,F}
		\end{bmatrix}
		& \text{if } u_0^\mathrm{n}>0,
		\vspace{1mm}
		\\
		\begin{bmatrix}
			\frac{1}{\sqrt{2}} \\
			\frac{1}{\sqrt{2}}\bm{n}_{K,F}
		\end{bmatrix}
		& \text{otherwise},
	\end{cases}
	\\
	\bm{\varLambda}^-_{K,F}
	& =
	\begin{cases}
		\begin{bmatrix}
			-c_0^- & 0                       \\
			0      & u_0^\mathrm{n} \bm{I}_2
		\end{bmatrix}
		& \text{if } u_0^\mathrm{n}<0,
		\vspace{1mm}
		\\
		\begin{bmatrix}
			-c_0^-
		\end{bmatrix}
		& \text{otherwise},
	\end{cases}
	&
	\bm{W}^-_{K,F}
	& =
	\begin{cases}
		\begin{bmatrix}
			\frac{1}{\sqrt{2}}              & \bm{0}_2^\dagger \\
			-\frac{1}{\sqrt{2}}\bm{n}_{K,F} & \bm{T}_{K,F}
		\end{bmatrix}
		& \text{if } u_0^\mathrm{n}<0,
		\\
		\vspace{1mm}
		\begin{bmatrix}
			\frac{1}{\sqrt{2}} \\
			-\frac{1}{\sqrt{2}}\bm{n}_{K,F}
		\end{bmatrix}
		& \text{otherwise},
	\end{cases}
\end{align}
where the subscripts $_{K,F}$ are avoided for brevity.
The \emph{numerical fluxes} then read
\begin{align}
	u_0^\mathrm{n}\widehat{p}_F + \rho_0c_0^2 (\bm{n}_{K,F}\cdot\widehat{\bm{u}}_F)
	& = \frac{1}{\sqrt{2}}\left(\sqrt{c_0^+} g^{+,\mathrm{n}}_{K,F}-\sqrt{c_0^-} g^{-,\mathrm{n}}_{K,F}\right),
	\\
	c_0\widehat{p}_F + \rho_0c_0u_0^\mathrm{n} (\bm{n}_{K,F}\cdot\widehat{\bm{u}}_F)
	& = \frac{1}{\sqrt{2}}\left(\sqrt{c_0^+} g^{+,\mathrm{n}}_{K,F}+\sqrt{c_0^-} g^{-,\mathrm{n}}_{K,F}\right),
	\\
	\rho_0c_0u_0^\mathrm{n} (\bm{\pi}_{K,F}\widehat{\bm{u}}_{F})
	& =
	\begin{cases}
		\sqrt{|u_0^\mathrm{n}|} \: \bm{T}_{K,F}\bm{g}^{+,\mathrm{t}}_{K,F} & \text{if } u_0^\mathrm{n}>0, \\
		\mathbf{0}_3                                                       & \text{if } u_0^\mathrm{n}=0, \\
		\sqrt{|u_0^\mathrm{n}|} \: \bm{T}_{K,F}\bm{g}^{-,\mathrm{t}}_{K,F} & \text{if } u_0^\mathrm{n}<0.
	\end{cases}
\end{align}
with the \emph{transmission variables} defined as
\begin{align}
	g^{+,\mathrm{n}}_{K,F}
	& := \sqrt{c_0^+/2} \: (p_K + \rho_0c_0\bm{n}_{K,F}\cdot\bm{u}_K),
	\\
	g^{-,\mathrm{n}}_{K,F}
	& :=
	\begin{cases}
		g^{+,\mathrm{n}}_{K',F}
		& \text{if } F\not\subset\partial\Omega,
		\\
		- \sqrt{c_0^-/c_0^+} \: g^{+,\mathrm{n}}_{K,F} + \sqrt{2c_0^-} \: s_{p}
		& \text{if } F\subset\Gamma_{p},
		\\
		\sqrt{c_0^-/c_0^+} \: g^{+,\mathrm{n}}_{K,F} - \sqrt{2c_0^-} \: \rho_0 c_0 \: s_u
		& \text{if } F\subset\Gamma_{u},
		\\
		\sqrt{c_0^-/2} \: s_{\mathrm{imp}}
		& \text{if } F\subset\Gamma_{\mathrm{imp}},
	\end{cases}
	\\
	\bm{g}^{+,\mathrm{t}}_{K,F}
	& := \rho_0c_0\sqrt{|u_0^\mathrm{n}|} \: \bm{T}^\dagger_{K,F}\bm{u}_K,
	&
	& \text{(if $u_0^\mathrm{n}>0$)}
	\\
	\bm{g}^{-,\mathrm{t}}_{K,F}
	& :=
	\begin{cases}
		\bm{g}^{+,\mathrm{t}}_{K',F}
		& \text{if } F\not\subset\Gamma_-,
		\\
		\rho_0c_0\sqrt{|u_0^\mathrm{n}|} \: \bm{s}_-
		& \text{if } F\subset\Gamma_-.
	\end{cases}
	&
	& \text{(if $u_0^\mathrm{n}<0$)}
\end{align}
By manipulating these relations, we can write the quantities $\widehat{p}_h$ and $\widehat{\bm{u}}_h$, which are called the \emph{numerical traces}, in terms of the physical quantities as
\begin{align}
	\widehat{p}_F
	& =
	\begin{cases}
		\frac{1}{2}(p_K+p_{K'}) + \frac{1}{2}\rho_0c_0 (\bm{n}_{K,F}\cdot(\bm{u}_K-\bm{u}_{K'}))
		& \text{if $F\not\subset\partial\Omega$},
		\\
		s_{p}
		& \text{if $F\subset\Gamma_{p}$,}
		\\
		p_K+\rho_0c_0 (\bm{n}_{K,F}\cdot\bm{u}_K-s_u)
		& \text{if $F\subset\Gamma_{u}$,}
		\\
		\frac{1}{2}( p_K + \rho_0c_0(\bm{n}_{K,F}\cdot\bm{u}_K) + s_{\mathrm{imp}} )
		& \text{if $F\subset\Gamma_{\mathrm{imp}}$,}
	\end{cases}
	\\
	\bm{n}_{K,F}\cdot\widehat{\bm{u}}_F
	& =
	\begin{cases}
		\frac{1}{2}\bm{n}_{K,F}\cdot(\bm{u}_K+\bm{u}_{K'}) + \frac{1}{2\rho_0c_0}(p_K-p_{K'})
		& \text{if $F\not\subset\partial\Omega$},
		\\
		(\bm{n}_{K,F}\cdot\bm{u}_K) + \frac{1}{\rho_0c_0}\left(p_K - s_{p}\right)
		& \text{if $F\subset\Gamma_{p}$,}
		\\
		s_u
		& \text{if $F\subset\Gamma_{u}$,}
		\\
		\frac{1}{2\rho_0c_0}\left(p_K+\rho_0c_0\bm{n}_{K,F}\cdot\bm{u}_K-s_{\mathrm{imp}}\right)
		& \text{if $F\subset\Gamma_{\mathrm{imp}}$,}
	\end{cases}
	\\
	\bm{\pi}_{K,F}\widehat{\bm{u}}_F
	& =
	\begin{cases}
		\bm{\pi}_{K,F}\bm{u}_K
		& \text{if $u_0^\mathrm{n}>0$,}
		\\
		\bm{\pi}_{K,F}\bm{u}_{K'}
		& \text{if $u_0^\mathrm{n}<0$ and $F\not\subset\Gamma_-$,}
		\\
		\bm{Ts}_-
		& \text{if $u_0^\mathrm{n}<0$ and $F\subset\Gamma_-$.}
	\end{cases}
\end{align}
Note that the tangential numerical trace $\bm{\pi}_{K,F}\widehat{\bm{u}}_F$ does not need to be defined if $u_0^\mathrm{n}=0$.

%%%%%%%%%%%%%%%%%%%%%%%%%%%%%%%%%%%%%%%%%%%%%%%%%%%%%%%%%%%%%%%%%%%%%%%%%%%%%%%%%%%%%%%%

\subsection{Hybridization with transmission variables}

To define the hybridized formulation, we introduce the functional spaces associated to the global outgoing and incoming transmission variables, respectively, as
\begin{align}
	\bm{\mathcal{G}}_h^+
	& := \textstyle \Big[\prod_{K\in\mathcal{T}_h}\prod_{F\in\mathcal{F}_K}\mathcal{P}_\mathrm{p}(F)\Big]
	\times \Big[\prod_{K\in\mathcal{T}_h}\prod_{F\in\mathcal{F}_K^+}\mathcal{P}^{2}_\mathrm{p}(F)\Big],
	\\
	\bm{\mathcal{G}}_h^-
	& := \textstyle \Big[\prod_{K\in\mathcal{T}_h}\prod_{F\in\mathcal{F}_K}\mathcal{P}_\mathrm{p}(F)\Big]
	\times \Big[\prod_{K\in\mathcal{T}_h}\prod_{F\in\mathcal{F}_K^-}\mathcal{P}^{2}_\mathrm{p}(F)\Big],
\end{align}
where $\mathcal{F}_K^+=\{F\in\mathcal{F}_K, u_0^\mathrm{n}>0\}$ and $\mathcal{F}_K^-=\{F\in\mathcal{F}_K, u_0^\mathrm{n}<0\}$ denote the collections of the faces of $K$ where the mean flow is outgoing and incoming, respectively.
The global outgoing and incoming variables are written in terms of their normal and tangential components as
\begin{align}
	\bm{G}_h^{+}
	= \begin{bmatrix}g_h^{+,\mathrm{n}} \\ \bm{g}_h^{+,\mathrm{t}}\end{bmatrix}
	\qquad\text{and}\qquad
	\bm{G}_h^{-}
	= \begin{bmatrix}g_h^{-,\mathrm{n}} \\ \bm{g}_h^{-,\mathrm{t}}\end{bmatrix},
\end{align}
respectively. The incoming variable $\bm{G}_h^{-} \in \bm{\mathcal{G}}_h^{-}$ is the unknown of the hybridized problem.
It is important to note that the number of transmission variables is not the same on each face, since it depends on the direction of the mean flow relative to the face normal.

The hybridized formulation, given by Problem \ref{pbm:red2_general}, can be explicitly written with the operators and the right-hand side vector defined as follows.

First, the \textit{exchange operator} $\Pi : \bm{\mathcal{G}}_h^{+}\rightarrow\bm{\mathcal{G}}_h^{-}$ is defined, for each face $F$ of each element $K$ and for all $\bm{G}^+_h \in \bm{\mathcal{G}}_h^{+}$, as
\begin{align}
	\Pi(\bm{G}^+_h)|_{K,F}
	=
	\begin{cases}
		\begin{bmatrix}
			\Pi_n(g_h^{+,\mathrm{n}})|_{K,F} \\
			\Pi_t(\bm{g}_h^{+,\mathrm{t}})|_{K,F}
		\end{bmatrix}
		& \text{if $u_0^\mathrm{n}<0$,}
		\\
		\Pi_n(g_h^{+,\mathrm{n}})|_{K,F}
		& \text{otherwise,}
	\end{cases}
\end{align}
with
\begin{align}
	\Pi_n(g_h^{+,\mathrm{n}})|_{K,F}      & =
	\begin{cases}
		g^{+,\mathrm{n}}_{K',F}
		& \text{if } F\not\subset\partial\Omega,    \\
		-g_{K,F}^{+,\mathrm{n}}\sqrt{c_0^-/c_0^+}
		& \text{if } F\subset\Gamma_{p},            \\
		g_{K,F}^{+,\mathrm{n}}\sqrt{c_0^-/c_0^+}
		& \text{if } F\subset\Gamma_{u},            \\
		0
		& \text{if } F\subset\Gamma_{\mathrm{imp}},
	\end{cases} \\
	\Pi_t(\bm{g}_h^{+,\mathrm{t}})|_{K,F} & =
	\begin{cases}
		\bm{g}^{+,\mathrm{t}}_{K',F}
		& \text{if } F\not\subset\Gamma_-, \\
		\mathbf{0}_2
		& \text{if } F\subset\Gamma_-.
	\end{cases}
	\qquad \text{(if $u_0^\mathrm{n}<0$)}
\end{align}

Then, the \textit{scattering operator} $\TS:\bm{\mathcal{G}}_h^{-}\rightarrow\bm{\mathcal{G}}_h^{+}$ is defined, for each face $F$ of each element $K$ and for all $\bm{G}^-_h \in \bm{\mathcal{G}}_h^{-}$, as
\begin{align}
	\TS(\bm{G}^{-}_h)|_{K,F}
	=
	\begin{cases}
		\begin{bmatrix}
			g^{+,\mathrm{n}}_{K,F} \\
			\bm{g}^{+,\mathrm{t}}_{K,F}
		\end{bmatrix}
		& \text{if $u_0^\mathrm{n}>0$,}
		\\
		g^{+,\mathrm{n}}_{K,F}
		& \text{otherwise,}
	\end{cases}
\end{align}
where $g^{+,\mathrm{n}}_{K,F}$ and $\bm{g}^{+,\mathrm{t}}_{K,F}$ are the outgoing transmission variables computed with the solution of the local problem corresponding to $K$, with the incoming transmission variable $\{\bm{G}^{-}_{K,F}\}_{F\in\mathcal{F}_K}$ contained in $\bm{G}^{-}_h$ as a given surface datum.
The local problem reads as follows.
\begin{problem}
	Find $(p_K,\bm{u}_K)\in\mathcal{P}_\mathrm{p}(K)\times\mathcal{P}^{2}_\mathrm{p}(K)$ such that, for all $(q_K,\bm{v}_K)\in\mathcal{P}_\mathrm{p}(K)\times\mathcal{P}^{2}_\mathrm{p}(K)$,
	\begin{align}
		\left\{
		\begin{aligned}
			& - \i\omega\big(p_K, q_K\big)_K
			- \big(p_K,\bm{u}_0\cdot\grad q_K\big)_K
			- \big(\rho_0c_0^2\bm{u}_K, \grad q_K\big)_K
			\\
			& \qquad\qquad
			+ \sum_F\Big\langle \frac{1}{\sqrt{2}}\sqrt{c_0^+} g^{+,\mathrm{n}}_{K,F}, q_K\Big\rangle_F
			= \sum_F\Big\langle \frac{1}{\sqrt{2}}\sqrt{c_0^-} g^{-,\mathrm{n}}_{K,F}, q_K\Big\rangle_F,
			\\
			& - \i\omega\big(\rho_0c_0\bm{u}_K, \bm{v}_K\big)_K
			- \big(\rho_0c_0\bm{u}_K,\grad\bm{v}_K\cdot\bm{u}_0\big)_K
			- \big(c_0p_K, \grad\cdot\bm{v}_K\big)_K
			\\
			& \qquad\qquad
			+ \sum_F\Big\langle \frac{1}{\sqrt{2}}\sqrt{c_0^+} g^{+,\mathrm{n}}_{K,F}, \bm{v}_K\Big\rangle_F
			+ \sum_{F,u_0^\mathrm{n}>0}\Big\langle \sqrt{|u_0^\mathrm{n}|} \: \bm{T}_{K,F}\bm{g}^{+,\mathrm{t}}_{K,F} , \bm{v}_K \Big\rangle_F
			\\
			& \qquad\qquad
			=
			- \sum_F\Big\langle \frac{1}{\sqrt{2}}\sqrt{c_0^-} g^{-,\mathrm{n}}_{K,F}, \bm{v}_K\Big\rangle_F
			- \sum_{F,u_0^\mathrm{n}<0}\Big\langle \sqrt{|u_0^\mathrm{n}|} \: \bm{T}_{K,F}\bm{g}^{-,\mathrm{t}}_{K,F} , \bm{v}_K \Big\rangle_F,
		\end{aligned}
		\right.
	\end{align}
	where $g^{+,\mathrm{n}}_{K,F}$ and $\bm{g}^{+,\mathrm{t}}_{K,F}$ are defined with the local fields $p_K$ and $\bm{u}_K$ using equations \eqref{eqn:aero:gPlusN} and \eqref{eqn:aero:gPlusT}.
\end{problem}

Finally, the \textit{right-hand side vector} $\bm{b}$ is defined, for each face $F$ of each element $K$, as
\begin{align}
	\bm{b}_{K,F}
	=
	\begin{cases}
		\begin{bmatrix}
			b^{n}_{K,F} \\
			\bm{b}^t_{K,F}
		\end{bmatrix}
		& \text{if $u_0^\mathrm{n}<0$,}
		\\
		b^{n}_{K,F}
		& \text{otherwise,}
	\end{cases}
\end{align}
with
\begin{align}
	b^{n}_{K,F}
	& =
	\begin{cases}
		0
		& \text{if }F\not\subset\partial\Omega,
		\\
		s_{p}\sqrt{2c_0^-}
		& \text{if } F\subset\Gamma_{p},
		\\
		-s_u \rho_0c_0\sqrt{2c_0^-}
		& \text{if } F\subset\Gamma_{u},
		\\
		s_{\mathrm{imp}}\sqrt{c_0^-/2}
		& \text{if } F\subset\Gamma_{\mathrm{imp}},
	\end{cases} \\
	\bm{b}^{t}_{K,F}
	& =
	\begin{cases}
		\mathbf{0}_2
		& \text{if } F\not\subset\Gamma_-,
		\\
		-\rho_0c_0\sqrt{|u_0^\mathrm{n}|}\bm{s}_-
		& \text{if } F\subset\Gamma_-.
	\end{cases}
	\qquad \text{(if $u_0^\mathrm{n}<0$)}
\end{align}

%%%%%%%%%%%%%%%%%%%%%%%%%%%%%%%%%%%%%%%%%%%%%%%%%%

\section{Numerical results}
\label{sec:numStudy}

In this section, we present two-dimensional numerical results to illustrate the method and to compare iterative procedures for solving the CHDG hybridized system.
Specifically, we consider the propagation of plane sound waves in a square domain (Section~\ref{sec:num1}), sound waves generated by a point source in a circular domain (Section~\ref{sec:num2}), and vorticity waves in a rigid waveguide (Section~\ref{sec:num3}) in the presence of a uniform mean flow.
These benchmarks allow us to test different types of waves and boundary conditions, and to study the influence of physical parameters.
The simulations are performed with dedicated \textsf{MATLAB} scripts.

%%%%%%%%%%%%%%%%%%%%%%%%%%%%%%%%%%%%%%%%%%%%%%%%%%%%%%%%%%%%%%%%%%%%%%%%%%%%%%%%%%%%%%%%

\subsection{Numerical setting}

The discretization is based on unstructured, conforming meshes of the domains composed of straight triangular elements.
Hierarchical shape functions based on Lobatto functions are used to represent the physical fields on the triangles and the transmission variables on the edges, see e.g.~\cite{solin2003, beriot2016}.
In all the cases, the polynomial degree is $\mathrm{p}=3$.
The meshes are constructed such that the relative numerical error is always close to $1\,\%$.
They are generated using \textsf{Gmsh} \cite{geuzaine2009} with a prescribed element size $h$.

Three iterative schemes are considered.
For the CHDG system, we use the fixed-point iteration, which converges since $\Pi\TS$ is a strict contraction.
Each iteration requires minimal computation and memory storage.
However, the number of iterations can vary greatly depending on the type of problem.
We also consider the CGNR \textit{(conjugate gradient normal residual)} method, which applies the conjugate gradient method to the normal system $\matalg{A}^\dagger\matalg{A}\vecalg{x}=\matalg{A}^\dagger\vecalg{b}$ for a given system $\matalg{A}\vecalg{x}=\vecalg{b}$.
The cost per iteration is higher because both the system matrix and its adjoint must be applied.
Finally, we consider the GMRES \textit{(generalized minimal residual)} iteration without restart.
Only one application of the matrix is required but the computational cost and the memory storage increase with each iteration.
In practice, a restart strategy should be employed to reduce the cost per iteration.
However, the objective here is to compare the GMRES method based on the number of iterations in the best-case scenario.

As in \cite{modave2023}, we use a symmetric preconditioning of the hybridized system to improve the convergence of the iterative solvers.
This is based on the matrix $\matalg{M}$ containing on its diagonal the local mass matrices associated with the faces of the elements.
Denoting the Cholesky factorization of this matrix as $\matalg{M} = \matalg{L}\matalg{L}^\top$, the precondcitioned system is $\tilde{\matalg{A}}\tilde{\vecalg{g}}=\tilde{\vecalg{b}}$ with $\tilde{\matalg{A}}:=\matalg{L}^{-1}\matalg{A}\matalg{L}^{-\top}$, $\tilde{\vecalg{g}}:=\matalg{L}^{\top}\vecalg{g}$ and $\tilde{\vecalg{b}}:=\matalg{L}^{-1}\vecalg{b}$.
This system corresponds to the one that would be obtained if orthonormal basis functions were used on the skeleton, see \cite{modave2023}.
It follows that the $2$-norm of an algebraic vector $\tilde{\vecalg{g}}$ is equal to the $L^2$-norm of the corresponding field, i.e.~$\|\tilde{\vecalg{g}}\|_2^2 = \tilde{\vecalg{g}}^\dagger\tilde{\vecalg{g}} = \vecalg{g}^\dagger\matalg{M}\vecalg{g} = \|\bm{g}_h\|_{\partial\mathcal{T}_h}^2$.

The different solution procedures are compared by tracking the evolution of the relative error over the course of the iterations.
The relative error is defined as
\begin{align}
	\text{relative error} := \frac{\|(p_h-p_\mathrm{ref},\bm{u}_h-\bm{u}_\mathrm{ref})\|_\mathcal{E}}{\|(p_\mathrm{ref},\bm{u}_\mathrm{ref})\|_\mathcal{E}}\;,
\end{align}
with the numerical solution $(p_h,\bm{u}_h)$ and the reference analytic solution $(p_\mathrm{ref},\bm{u}_\mathrm{ref})$, where $\|\cdot\|_\mathcal{E}$ is the energy norm based on the energy density \eqref{eq:aero:energy} defined as
\begin{align}
	\|(p,\bm{u})\|^2_\mathcal{E}
	& := \sum_{K\in\mathcal{T}_h}
	\Big( \frac{1}{4\rho_0 c_0^2} \|p\|^2_{L^2(K)} + \frac{1}{4}\rho_0 \|\bm{u}\|^2_{L^2(K)}\Big).
\end{align}

%%%%%%%%%%%%%%%%%%%%%%%%%%%%%%%%%%%%%%%%%%%%%%%%%%%%%%%%%%%%%%%%%%%%%%%%%%%%%%%%%%%%%%%%

\subsection{Sound plane wave}
\label{sec:num1}

We consider the propagation of a plane wave in a uniform mean flow $\bm{u}_0 = u_0 \bm{\hat{u}}_0$, where $\bm{\hat{u}}_0$ is a unit vector, and no volume source term, i.e.~$s_\mathrm{vol}=0$.
The reference fields read
\begin{align}
	p_\mathrm{ref}(\bm{x}) = e^{\i\kappa\bm{\hat{\kappa}}\cdot\bm{x}}
	\qquad\text{and}\qquad
	\bm{u}_\mathrm{ref}(\bm{x}) & = \bm{\hat{\kappa}} e^{\i\kappa\bm{\hat{\kappa}}\cdot\bm{x}} / (\rho_0c_0)
\end{align}
with the wavenumber $\kappa := \omega/(c_0 + \bm{\hat{\kappa}}\cdot\bm{u}_0)$ and the propagation direction $\bm{\hat{\kappa}}$, which is a unit vector.
The simulations are performed using a square computational domain $\Omega = {]0,1[}^2$, a mean flow velocity with amplitude $u_0 = 0.25$ and direction $\bm{\hat{u}}_0 = (1,1)/\sqrt{2}$, and density $\rho_0=1$.
Two angular frequencies $\omega=15\pi$ and $\omega=25\pi$ are tested.

Several configurations are considered for the propagation direction and the boundary conditions.
Both the propagation in the same direction as the mean flow (i.e.~$\bm{\hat{\kappa}} = \bm{\hat{u}}_0$) and in the opposite direction (i.e.~$\bm{\hat{\kappa}} = -\bm{\hat{u}}_0$) are tested, as illustrated in Figures~\ref{fig:bench1:conf1} and \ref{fig:bench1:conf2}.
The sound speed is $c_0=1$ in the first configuration and $c_0=1.5$ in the second.
The solutions are then identical in both cases, even though the sets of equations differ.
A reference solution is shown on Figure~\ref{fig:bench1:solRef}.
In accordance with Assumption \ref{hyp:aero:passivity}, both the inflow boundary condition on $\bm{T}^\dagger\bm{u}$ and an impedance condition are prescribed on the region of the boundary where the mean flow is ingoing, i.e.~on $\Gamma_-$.
On the other region, i.e.~on $\Gamma_+$, a boundary condition on pressure or an impedance condition is prescribed.
In all cases, the boundary conditions are non-homogeneous, and the surface data are defined using the reference fields.

\begin{figure}[!tbh]
	\begin{subfigure}{0.30\textwidth}
		\centering
		\includegraphics{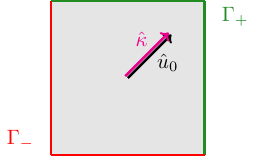}
		\centering
		\caption{Configuration 1}
		\label{fig:bench1:conf1}
	\end{subfigure}
	\hfill
	\begin{subfigure}{0.30\textwidth}
		\centering
		\includegraphics{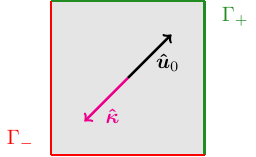}
		\centering
		\caption{Configuration 2}
		\label{fig:bench1:conf2}
	\end{subfigure}
	\hfill
	\begin{subfigure}{0.34\textwidth}
		\centering
		\includegraphics[width=27mm]{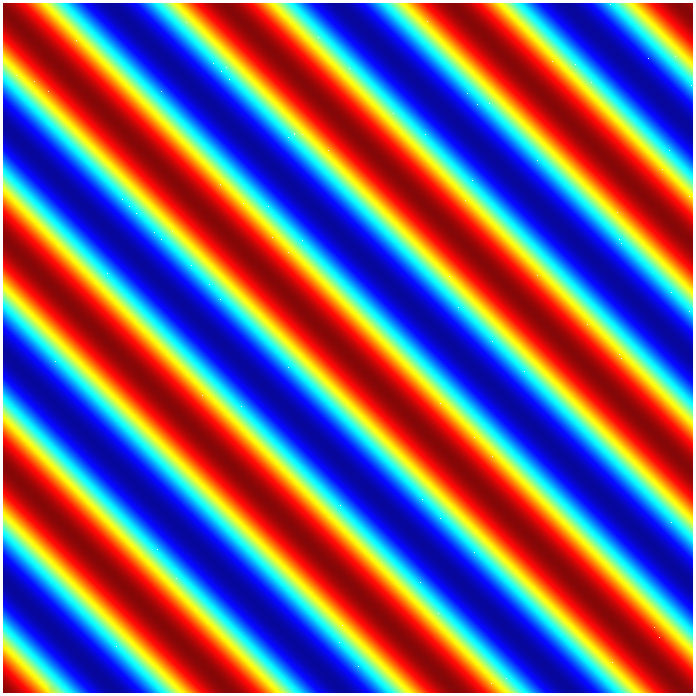}
		\includegraphics[width=5mm]{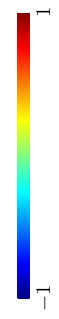}
		\caption{Real part of the reference field $p_\text{ref}$.}
		\label{fig:bench1:solRef}
	\end{subfigure}
	\caption{Plane wave benchmark. Considered configurations with directions of propagation of the plane wave $\bm{\hat{\kappa}}$ and of the background flow $\bm{\hat{u}}_0$ in two cases. Reference solution.}
	\label{fig:bench1:config}
\end{figure}

\begin{figure}[!tbhp]
	\centering
	\includegraphics{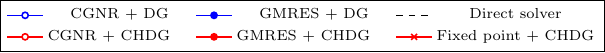}
	\\
	\medskip
	\begin{subfigure}[b]{0.49\textwidth}
		\centering
		\caption{$\bm{\hat{\kappa}}=\bm{\hat{u}}_0$ ; $\omega=15\pi$ ; $c_0=1$ ; $u_0=0.25$ ; $h=1/13$}
		\includegraphics{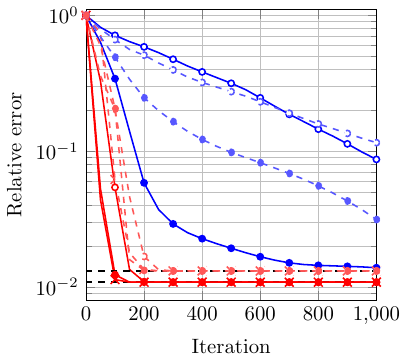}
	\end{subfigure}
	\hfill
	\begin{subfigure}[b]{0.49\textwidth}
		\centering
		\caption{$\bm{\hat{\kappa}}=-\bm{\hat{u}}_0$ ; $\omega=15\pi$ ; $c_0=1.5$ ; $u_0=0.25$ ; $h=1/13$}
		\includegraphics{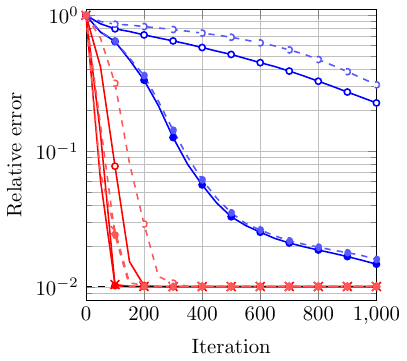}
	\end{subfigure}
	\medskip
	\begin{subfigure}[b]{0.49\textwidth}
		\centering
		\caption{$\bm{\hat{\kappa}}=\bm{\hat{u}}_0$ ; $\omega=25\pi$ ; $c_0=1$ ; $u_0=0.25$ ; $h=1/22$}
		\includegraphics{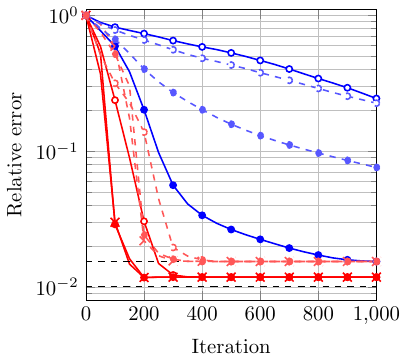}
	\end{subfigure}
	\hfill
	\begin{subfigure}[b]{0.49\textwidth}
		\centering
		\caption{$\bm{\hat{\kappa}}=-\bm{\hat{u}}_0$ ; $\omega=25\pi$ ; $c_0=1.5$ ; $u_0=0.25$ ; $h=1/22$}
		\includegraphics{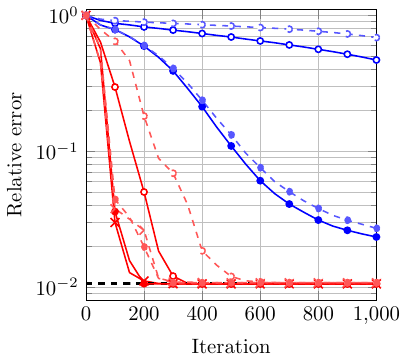}
	\end{subfigure}
	\caption{Plane wave benchmark. Error history with different iterative schemes applied to the DG and CHDG systems for the configuration with an impedance condition on $\Gamma_+$.
		The dashed light red and light blue curves correspond to the configuration with a condition on $p$ on $\Gamma_+$.
		The dashed horizontal lines correspond to the relative numerical errors obtained using a direct solver.}
	\label{fig:histo:planewave}
\end{figure}

Figure~\ref{fig:histo:planewave} shows the history of the relative error for the different configurations and iterative procedures.
The following observations can be made.
The fixed-point iteration applied to the CHDG system always converges, which is consistent with the theoretical result that $\Pi\TS$ is a strict contraction.

The performance of the iterative schemes can be compared in terms of number of iterations.
The fixed-point iteration and GMRES applied to the CHDG system produce very similar error decay while the CGNR scheme is slightly slower.
In all the cases, the error decay is slower for larger wavenumbers $\kappa$ and finer meshes.
The CGNR and GMRES schemes always converge faster when applied to the CHDG system than to the DG system.
The three iterative schemes are slower when solving the DG system compared to the CHDG system.
These observations are consistent with those obtained in \cite{modave2023} for the Helmholtz case.

The error history changes slightly when the impedance condition is replaced with a condition on pressure on $\Gamma_+$ for a given mesh, which corresponds to the dashed lines on Figure~\ref{fig:histo:planewave}.
The impact of the boundary condition is much less evident than it was in \cite{modave2023}, where the error decay was much slower for cavity cases with only a boundary condition on pressure.
Here, with mean flow, we must always prescribe a boundary condition on the inflow region that leads to dissipation.
Therefore, unlike in the Helmholtz case, we can never consider a fully dissipation-free situation.

%%%%%%%%%%%%%%%%%%%%%%%%%%%%%%%%%%%%%%%%%%%%%%%%%%%%%%%%%%%%%%%%%%%%%%%%%%%%%%%%%%%%%%%%

\subsection{Sound point source}
\label{sec:num2}

We consider the wave field generated by a point source $s_\mathrm{vol}(\bm{x})=\delta(\bm{x}-\bm{x}_s)$ at position $\bm{x}_s$ in a uniform mean flow $\bm{u}_0 = u_0\bm{e}_x$, where $\bm{e}_x$ is the unit vector in the $x$-direction.
The reference solution reads
\begin{align}
	p_\mathrm{ref}(\bm{x})
	& = (-\i\omega+\bm{u}_0\cdot\grad)\mathcal{G}(\bm{x}-\bm{x}_s), \\
	\bm{u}_\mathrm{ref}(\bm{x})
	& = -\grad\mathcal{G}(\bm{x}-\bm{x}_s)/\rho_0,
\end{align}
where the Green function $\mathcal{G}$ is the solution of the equation $(-\i\omega+\bm{u}_0\cdot\grad)^2\mathcal{G}-c_0^2\Delta \mathcal{G} = \delta(\bm{x})$ in $\mathbb{R}^2$, see \cite{bailly2000numerical}.
Its expression is given by
\begin{align}
	\mathcal{G}(\bm{x})
	= \frac{\i}{4c_0\sqrt{c_0^2-u_0^2}}
	\: \mathrm{H}_0^{(1)}\left(\omega\frac{\sqrt{c_0^2x^2+(c_0^2-u_0^2)y^2}}{c_0^2-u_0^2}\right)
	\exp\Big(-\i\omega\frac{u_0}{c_0^2-u_0^2} x\Big)\;,
\end{align}
where $\bm{x}=(x,y)$ and $\mathrm{H}_0^{(1)}$ is the Hankel function of the first kind and order zero.

The computational domain is the unit disc $\Omega=\{\bm{x} : \|\bm{x}\|<1\}$.
The point source is placed at $\bm{x}_s=(-u_0/c_0,0)$, so that the radiating waves hit the boundary with normal incidence.
We impose the impedance condition $p-\rho_0c_0\bm{n}\cdot\bm{u}=0$ on $\partial\Omega$ and the inflow condition $\bm{T}^\dagger\bm{u} = \bm{T}^\dagger\bm{u}_\mathrm{ref}$ on the inflow boundary $\Gamma_-$.
The physical parameters are $\rho_0=1$, $c_0=1$ and $\omega=40$.
The mean flow velocities $u_0 = 0.25$ and $0.75$ are tested.
Figure~\ref{fig:bench2:config} shows the setting together with examples of the mesh and of the reference solution.

\begin{figure}[!htb]
	\centering
	\begin{subfigure}[b]{0.3\textwidth}
		\centering
		\includegraphics{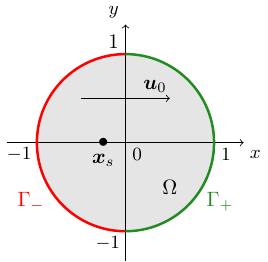}
		\caption{Computational domain.}
	\end{subfigure}
	\centering
	\begin{subfigure}[b]{0.3\textwidth}
		\centering
		\includegraphics[width=0.7\linewidth]{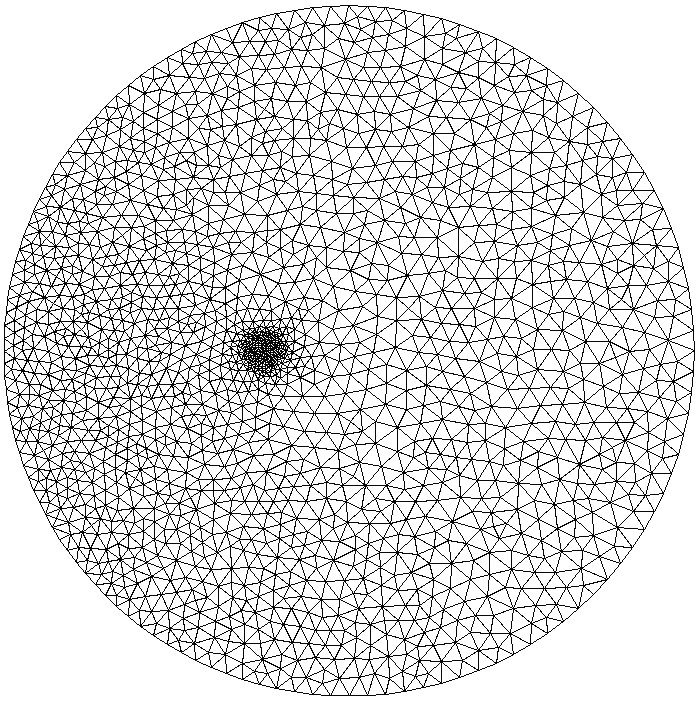} \\
		\phantom{\centering
			\vspace{+2cm}
			\centering}
		\caption{Adapted mesh.}
	\end{subfigure}
	\centering
	\begin{subfigure}[b]{0.3\textwidth}
		\centering
		\includegraphics[width=0.7\linewidth]{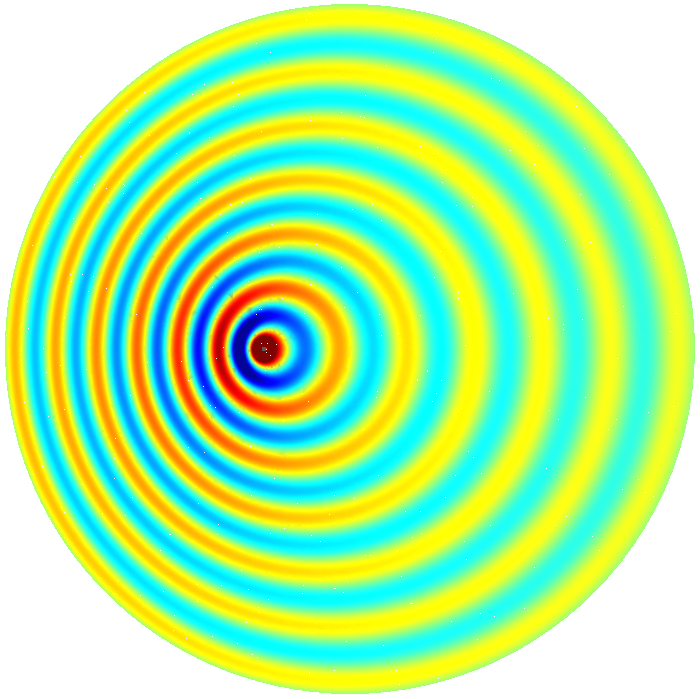} \\
		\centering
		\includegraphics[width=0.7\linewidth]{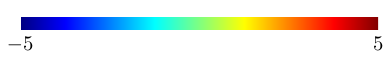}
		\centering
		\caption{Real part of $p_\text{ref}$.}
	\end{subfigure}
	\caption{Source point benchmark. Setting, mesh and reference solution for $u_0=0.25$.}
	\label{fig:bench2:config}
\end{figure}

The effect of the mean flow velocity $u_0$ is to shorten the acoustic wavelength upstream of the source, and lengthen it downstream, as can be seen in Figure~\ref{fig:bench2:config}(c).
In addition, the use of a point source leads to a singular solution at the source position, which is more difficult to approximate than the plane wave solution of the previous section.
The local element size in the mesh is therefore refined near the source and adjusted throughout the domain to capture the variation of the local wavelength, given in this case by
\begin{equation}
	\lambda(\bm{x}) = \frac{2\pi}{\omega}[c_0 - \bm{u}_0\cdot(\bm{x}-\bm{x}_s)/\|\bm{x}-\bm{x}_s\|]
	\;,
\end{equation}
The element size is set to $\tilde{h}/5$ in the vicinity of the source, and it is $\tilde{h}\lambda/\lambda_0$ elsewhere, where $\tilde{h}$ is a prescribed average mesh size and $\lambda_0=2\pi c_0/\omega$ is the wavelength in the absence of mean flow.
The average mesh size $\tilde{h}$ is selected to maintain a relative error close to $1\,\%$.
Because the solution is singular at the source position, the elements closest to $\bm{x}_s$, namely those with at least one of its vertices within a distance $2\tilde{h}$ from $\bm{x}_s$, are excluded when computing the error and solution norms.

\begin{figure}[!tb]
	\centering
	\includegraphics{main_8.pdf}
	\\
	\medskip
	\begin{subfigure}[b]{0.49\textwidth}
		\centering
		\caption{$u_0=0.25$ ; $\tilde{h}=1/20$}
		\includegraphics{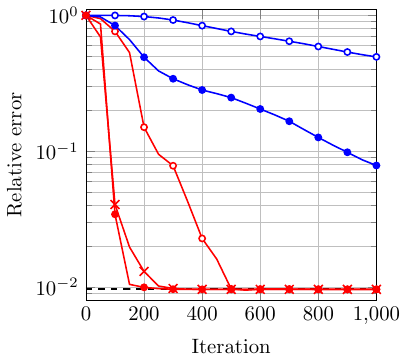}
	\end{subfigure}
	\hfill
	\begin{subfigure}[b]{0.49\textwidth}
		\centering
		\caption{$u_0=0.75$ ; $\tilde{h}=1/40$}
		\includegraphics{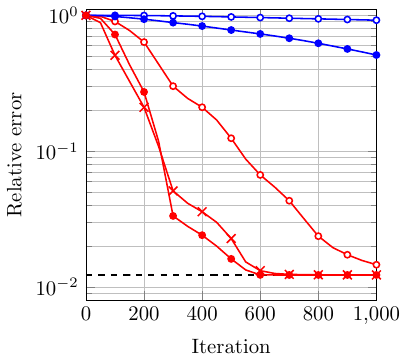}
	\end{subfigure}
	\caption{Point source benchmark. Error history with different iterative schemes applied to the DG and CHDG systems.
		The dashed horizontal lines correspond to the relative numerical errors obtained using a direct solver.}
	\label{fig:bench2:resu}
\end{figure}

Figure~\ref{fig:bench2:resu} shows the history of the relative error for the two configurations and the three iterative procedures.
As with the previous benchmark, CGNR and GMRES perform significantly better when applied to the CHDG system than to the DG system.
The convergence is slower with the largest flow velocity $u_0=0.75$.
Nevertheless, in such a case, the error decay is still relatively quick with the fixed point and GMRES applied to the CHDG system, while the iterative solution of the DG system is particularly slow with all three methods.
Here, the fixed point appears to be slightly slower than GMRES when solving the CHDG system.

%%%%%%%%%%%%%%%%%%%%%%%%%%%%%%%%%%%%%%%%%%%%%%%%%%%%%%%%%%%%%%%%%%%%%%%%%%%%%%%%%%%%%%%%

\subsection{Vorticity wave}
\label{sec:num3}

As mentioned above, the linearized Euler equations support two types of waves: sound waves and vorticity waves.
We now assess the CHDG method for the vorticity waves, by considering an infinite rigid duct with a uniform mean flow and no volume source term.
The duct walls are along the $x$-direction and the duct height is unity.
% , i.e.~$\Omega_\infty=\mathbb{R}\times{]0,1[}$.
As before, the mean flow velocity is $\bm{u}_0 = u_0\bm{e}_x$, with $u_0>0$.
The vorticity waves in this problem can be written in terms of the normal modes:
\begin{align}
	p_{\mathrm{ref},n}(\bm{x})      & = 0,                                                                                 \\
	\bm{u}_{\mathrm{ref},n}(\bm{x}) & = \left( (\i n\pi u_0/\omega) \cos(n\pi y), \sin(n\pi y)\right) e^{\im\omega x/u_0},
\end{align}
in which $\bm{x} = (x,y)$ and $n\in\mathbb{N}$ is the mode number.
While we had $\curl\bm{u}=\mathbf{0}$ for the sound waves considered before, we have $\div\bm{u}=0$ for the vorticity waves.

For the numerical simulations, a chosen mode $n$ is computed within the domain $\Omega = {]0,2[}\times{]0,1[}$, see Figure~\ref{fig:bench3:domain}.
The rigid wall condition $\bm{n}\cdot\bm{u}=0$ is prescribed on the upper and lower walls $\Gamma_\mathrm{wall}$.
In accordance with Assumption \ref{hyp:aero:passivity}, both the inflow condition $\bm{T}^\dagger\bm{u} = \bm{T}^\dagger\bm{u}_{\mathrm{ref},n}$ and the impedance condition $p-\rho_0c_0\bm{n}\cdot\bm{u} = p_{\mathrm{ref},n}-\rho_0c_0\bm{n}\cdot\bm{u}_{\mathrm{ref},n}$ are applied on the inflow boundary $\Gamma_\mathrm{in}$.
Either the impedance condition or the condition $p = p_{\mathrm{ref},n}$ is prescribed on the outflow boundary $\Gamma_\mathrm{out}$.
The physical parameters are $\rho_0=1$, $c_0=1$, $\omega=5\pi$ and $u_0=0.5$.
We consider the modes $n=1$ and $10$, which are shown in Figures~\ref{fig:bench3:solRef1} and \ref{fig:bench3:solRef10}, respectively.
A finer mesh is required for the higher-order mode, which is more oscillatory, in order to keep the relative error close $1\,\%$.

\begin{figure}[!htb]
	\centering
	\begin{subfigure}[b]{0.9\textwidth}
		\centering
		\includegraphics{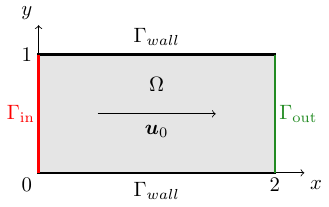}
		\centering
		\caption{Bounded rigid duct with a horizontal right-propagating background flow.}
		\label{fig:bench3:domain}
	\end{subfigure}
	\\
	\medskip
	\begin{subfigure}[b]{0.49\textwidth}
		\centering
		\includegraphics[width=0.6\linewidth]{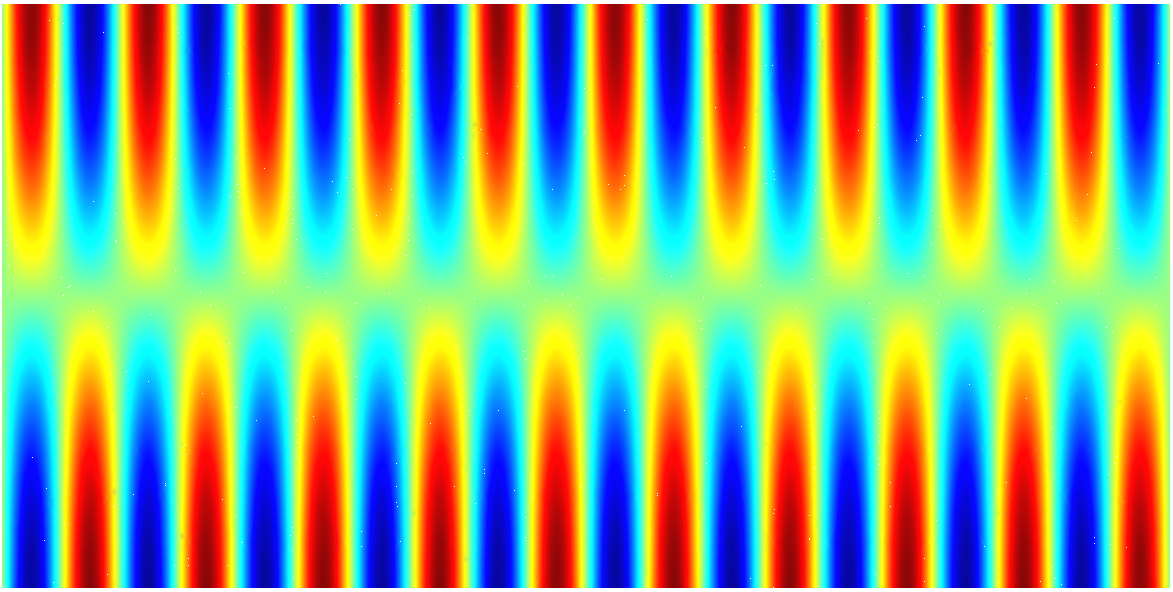} \\
		\centering
		\includegraphics[width=50mm]{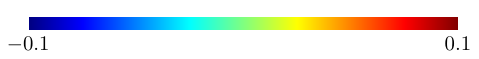}
		\centering
		\caption{Real part of $u_{\text{ref},x}$ for $n=1$ and $u_0=0.5$.}
		\label{fig:bench3:solRef1}
	\end{subfigure}
	\centering
	\begin{subfigure}[b]{0.49\textwidth}
		\centering
		\includegraphics[width=0.6\linewidth]{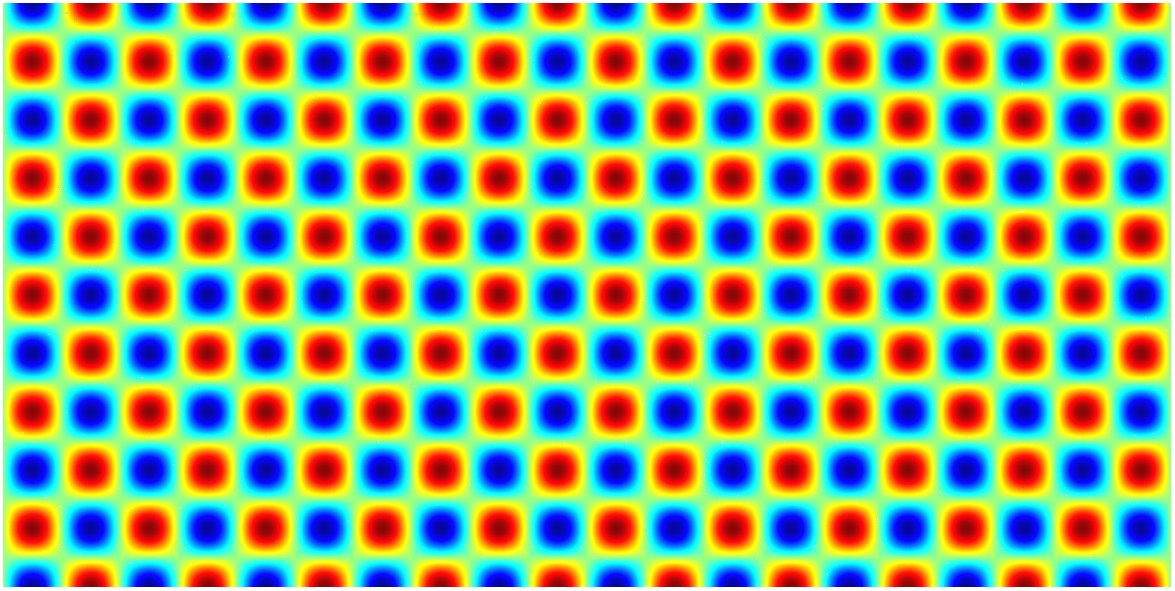} \\
		\centering
		\includegraphics[width=50mm]{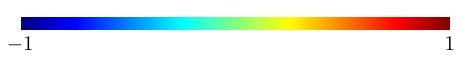}
		\centering
		\caption{Real part of $u_{\text{ref},x}$ for $n=10$ and $u_0=0.5$.}
		\label{fig:bench3:solRef10}
	\end{subfigure}
	\caption{Vorticity wave benchmark. Domain and reference solution.}
\end{figure}

Figure~\ref{fig:bench3:resu} shows the history of the relative error for the different cases and iterative procedures.
The results confirm that the observations made for the benchmarks with sound waves are also valid for cases with vorticity waves.
The fixed point applied to the CHDG system always converges.
Iterative schemes applied to the CHDG system converge more quickly than those applied to the DG system.
In the case of the CHDG system, the decay rates are similar for the fixed-point and GMRES iterations, whereas CGNR is slightly slower.
The results are very similar, and sometimes identical, when prescribing an impedance condition or a condition on pressure on $\Gamma_\mathrm{out}$.

\begin{figure}[!tbp]
	\centering
	\includegraphics{main_8.pdf}
	\\
	\begin{subfigure}[b]{0.49\textwidth}
		\centering
		\caption{$n=1$ ; $h=1/11$}
		\includegraphics{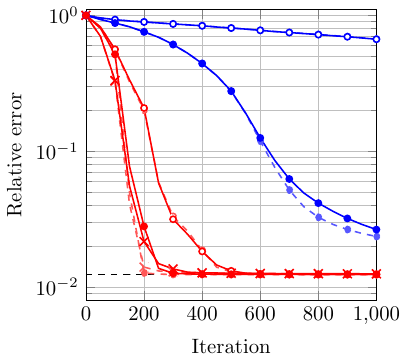}
	\end{subfigure}
	\hfill
	\begin{subfigure}[b]{0.49\textwidth}
		\centering
		\caption{$n=10$ ; $h=1/20$}
		\includegraphics{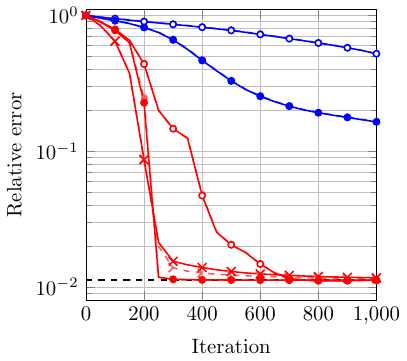}
	\end{subfigure}
	\caption{Vorticity wave benchmark. Error history with different iterative schemes applied to the DG and CHDG systems for the configuration with an impedance condition on $\Gamma_\mathrm{out}$.
		The dashed light red and light blue curves correspond to the configuration with a condition on $p$ on $\Gamma_\mathrm{out}$.
		The dashed horizontal lines correspond to the relative numerical errors obtained using a direct solver.}
	\label{fig:bench3:resu}
\end{figure}

%%%%%%%%%%%%%%%%%%%%%%%%%%%%%%%%%%%%%%%%%%%%%%%%%%

\section{Conclusion}
\label{conclusion}

The CHDG approach, initially studied for the Helmholtz equation in \cite{modave2023}, has been extended to general wave propagation problems defined with the time-harmonic version of symmetric hyperbolic systems with constant coefficients \cite{pescuma2025b}.
This class of problems covers a broad range of applications, including the propagation of acoustic and electromagnetic waves in homogeneous media, which are both addressed in \cite{modave2023} and \cite{rappaport2025}, respectively.

The method relies on a discontinuous Galerkin scheme with upwind numerical fluxes that are written in terms of transmission variables in a general framework.
These variables can be interpreted as physical quantities that are transported outward and inward across each face of the elements and are intuitively connected to the notion of wave propagation.
Incoming transmission variables are used to define the hybridized CHDG system, which can be written as $(\TI-\Pi\TS)\bm{G}^-=\bm{b}$.
Assuming that the boundaries are passive, it is proven that the $\Pi\TS$ is a strict contraction.

The framework is applied to the time-harmonic linearized Euler equations with a uniform subsonic mean flow.
The key elements of the scheme are explicitly derived for this model, which has the particularity of supporting both acoustic and vorticity waves.
Unlike the acoustic and electromagnetic models previously considered, the sign of the eigenvalues of the flux matrix at the element faces depends on the orientation of the mean flow.
Consequently, the number of boundary conditions and incoming transmission variables on each face can vary throughout the domain.
The proposed general framework handles this aspect effectively.

We validated and studied the method using 2D numerical benchmarks involving sound and vorticity waves in uniform mean flows and different types of boundary conditions.
The results show that the fixed point always converges when solving the CHDG system, which is in accordance with the theory.
For the considered benchmarks, the fixed point is nearly as fast as GMRES and always faster than CGNR.
These conclusions are consistent with those in \cite{modave2023} and \cite{rappaport2025}, in which it was shown that solving the CHDG system iteratively requires fewer iterations than solving standard HDG and DG systems with classical schemes.
However, since an equivalent HDG scheme is not available here, the comparison is limited to CHDG and DG.
As expected, the iterative solution of the DG system always requires more iterations.

Since the proposed numerical results were obtained using a \textsf{MATLAB} script, conducting a complete performance study based on runtime and memory storage would be irrelevant.
Nevertheless, preliminary results in the context of electromagnetic waves are presented in \cite{rappaport2025}, and we are currently developing a parallel GPU-accelerated code.

Future work will apply the proposed general framework to other physical models, such as elastic waves and wave propagation in anisotropic media.
The framework can be extended to non-homogeneous media using the strategies studied in \cite{pescuma2025}.

\paragraph{Acknowledgments.}
This work was supported in part by the ANR JCJC project \emph{WavesDG} (research grant ANR-21-CE46-0010).

%%%%%%%%%%%%%%%%%%%%%%%%%%%%%%%%%%%%%%%%%%%%%%%%%%

\footnotesize
\setlength{\bibsep}{0pt plus 0ex}
\bibliographystyle{abbrvnat}
\bibliography{myrefs}
\addcontentsline{toc}{section}{References}

\end{document}